\documentclass[12pt]{article}

\usepackage{graphicx}
\usepackage{epstopdf, epsfig}
\usepackage{listings}
\usepackage{fancyhdr}
\usepackage{amsmath}
\makeatletter
\def\maketag@@@#1{\hbox{\m@th\normalfont\normalsize#1}}
\makeatother
\usepackage{amsfonts}
\usepackage{amssymb}
\usepackage{amsthm}
\usepackage[latin1]{inputenc}
\usepackage{array}
\usepackage{comment}
\usepackage{geometry}
\newtheorem{thm}{Theorem}[section]

\newtheorem{prop}[thm]{Proposition}
\numberwithin{equation}{section}

\begin{document}

\begin{center}\Large\textbf{Analytical solution of the incompressible Navier-Stokes equations}\end{center}
\normalsize
\textbf{F. Salmon}\\
UMR CNRS 5295 I2M, Bordeaux University, France\\
Fabien.Salmon@u-bordeaux.fr


\begin{abstract}
The Navier-Stokes equations, which govern fluid motions, are not resolved yet. This investigation relates to the application of the power series method to the incompressible Navier-Stokes equations. This method involves replacing variables by power series in the equations. The resulting relations allow the computation of the series coefficients from the initial and boundary conditions. All the steps of applying the method to the Navier-Stokes equations are detailed and the exact analytical solution is then established. The last part discusses the domain of convergence of the solution.
\end{abstract}

%
\vspace{2pc}

\section{Introduction}
The Navier-Stokes equations describe the fluid motion. They were named in honor of Claude-Louis Navier and George Gabriel Stokes. At present, the general solution of this set of equations remains unknown. Except in specific cases with analytical solution, fluid motion problems request high performance computing. Besides, the Navier-Stokes existence and smoothness problem belongs to the Millennium Prize Problems.
\newline \newline
Many studies handle the existence and uniqueness of solutions to the Navier-Stokes equations. The existence of weak solutions to the incompressible case were proved in \cite{Leray}. Moreover, the weak solutions of the compressible equations were analysed with the purpose of characterizing them \cite{B3}, \cite{B4}. Regarding strong solutions, numerous studies investigated their existence as well as their properties \cite{B1}, \cite{B5}. Likewise, scientists studied solutions of particular problems with specified boundary conditions \cite{B2}. \newline \newline
This article features the formulation of a general analytical solution to the incompressible Navier-Stokes equations. The resolution is based on the power series method \cite{livre} which involves replacing variables by power series in the equations. The power series coefficients are then determined by solving the recurrence relations. This approach provides a power series solution as long as it exists. \newline \newline
First, an effortless 2D differential equation gives an overview of the method. To make clear the restriction of this technique, the differences between this approach and a general one applicable to the simple example are indicated. Second, the whole process of applying the method to the Navier-Stokes equations is outlined. The resulting analytical solution is then delivered. This publication also contains a discussion about the domain of convergence.

\section{Description of the power series method}
\label{Method}

The power series method enables the construction of power series solutions to differential equations. The approach is based on replacing variables by power series in the equations. Then, the power series coefficients are computed by the reformulated equations. This mathematical technique is often adopted for one-dimensional differential equations but rarely for multidimensional equations. The complexity of calculations actually increases with the number of dimensions but the method remains relevant. This section outlines the power series method through the resolution of the elementary two-dimensional differential equation
\begin{equation}
\frac{\partial f}{\partial x}-\frac{\partial f}{\partial y}=0
\label{EqDiff}
\end{equation}

The solution is sought in $\mathbb{R}^+\times\mathbb{R}^+$. 
Before solving the equation with power series, the resolution is performed through a substitution. The substitution $\psi(x,y)=(u,v)=(x,x+y)$ yields (with $f(x,y)=g(u,v)$)
\begin{equation}
\left\lbrace
\begin{array}{lc}
\frac{\partial f}{\partial x}=\frac{\partial g}{\partial u}\frac{\partial u}{\partial x}+\frac{\partial g}{\partial v}\frac{\partial v}{\partial x}=\frac{\partial g}{\partial u}+\frac{\partial g}{\partial v}\\
\frac{\partial f}{\partial y}=\frac{\partial g}{\partial u}\frac{\partial u}{\partial y}+\frac{\partial g}{\partial v}\frac{\partial v}{\partial y}=\frac{\partial g}{\partial v}\\
\end{array}\right.
\end{equation}
Then, the differential equation becomes
\begin{equation}
\frac{\partial g}{\partial u}=0
\end{equation}
Thus, the solutions to the equation \eqref{EqDiff} are the functions depending on $x+y$: $f(x,y)=g(u,v)=\phi(v)=\phi(x+y)$. \newline
Now, the power series method is applied to \eqref{EqDiff}. The function $f$ is assumed to be of the form
\begin{equation}
f(x,y)=\sum_{n=0}^\infty\sum_{m=0}^\infty a_{n,m}x^ny^m
\end{equation}
The substitution of the power series expansion of $f$ in \eqref{EqDiff} yields
\begin{equation}
a_{n+1,m}=\frac{m+1}{n+1}a_{n,m+1}
\label{coef}
\end{equation}
The equation \eqref{coef} corresponds to a recurrence relation with two indices. The coefficients must be expressed from the boundary conditions coefficients ($n=0$ or $m=0$). The resolution of this relation gives
\begin{equation}
a_{n,m}=\frac{(m+n)!}{n!m!}a_{0,m+n}
\end{equation}
So, the power series solution of the differential equation \eqref{EqDiff} is
\begin{equation}
f(x,y)=\sum_{n=0}^\infty\sum_{m=0}^\infty {m+n\choose m} a_{0,m+n}x^ny^m
\label{Eqint}
\end{equation}
The solution \eqref{Eqint} does not look like to the one formerly obtained and some modifications have to be achieved. By setting $k=m+n$,
\begin{equation}
f(x,y)=\sum_{k=0}^\infty\sum_{m=0}^k {k\choose m} a_{0,k}x^{k-m}y^m
\end{equation}
The binomial theorem simplifies the prior expression:
\begin{equation}
f(x,y)=\sum_{k=0}^\infty a_{0,k}(x+y)^k
\end{equation}
The solutions are still the functions such that $f(x,y)=\phi(x+y)$ where $\phi$ derived from the boundary conditions. The coefficients $a_{0,k}$ correspond to the power series coefficients of $y\mapsto f(0,y)$:
\begin{equation}
f(0,y)=\sum_{m=0}^\infty a_{0,m}y^m
\end{equation}
The prescribed boundary conditions must be given by a power series so that the solution can also be of this form. This constraint constitutes a prerequisite for multidimensional problems. Then, the domain of convergence of the series solution depends on the boundary conditions.
First, an exponential boundary condition is assumed: $f(0,y)=e^y$. So, necessarily $a_{0,m}=\frac{1}{m!}$. The power series solution is then 
\begin{equation}
f_1(x,y)=\sum_{k=0}^\infty \frac{1}{k!}(x+y)^k=e^{x+y}
\label{sol1}
\end{equation}
The domain of convergence of the power series is $\mathbb{R}^2$. Therefore, \eqref{sol1} is the global solution.\newline
Now, $f(0,y)=\frac{1}{y+1}$ is imposed. In such a case, $a_{0,m}=(-1)^m$. The power series solution is
\begin{equation}
f_2(x,y)=\sum_{k=0}^\infty (-1)^k (x+y)^k
\end{equation}
The domain of convergence is $D=\{(x,y): |x+y|< 1\}$ in spite of the solution being clearly $(x,y)\mapsto \frac{1}{x+y+1}$ on $\mathbb{R^+}\times\mathbb{R^+}$.\newline\newline
In short, the power series method only provides power series solutions to differential equations. For multidimensional problems, the power series method requires boundary conditions of the form of power series. This restriction is not very constraining regarding the Navier-Stokes equations because the fluid boundary conditions often satisfied this point.
Furthermore, the domain of convergence of the series solutions to multidimensional equations depends on the boundary conditions. 

\section{Application to the incompressible Navier-stokes equations}
This section concerns the application of the power series method to the incompressible Navier-Stokes equations:
\begin{equation}
\left\lbrace
\begin{array}{lc}
\nabla . \mathbf{U} = 0\\
\rho\left(\frac{\partial\mathbf{U}}{\partial t}+\mathbf{U}.\nabla\mathbf{U}\right)=-\nabla P + \mu \nabla^2 \mathbf{U}+\rho\mathbf{g}
\end{array}\right.
\label{NS}
\end{equation}
$\mathbf{U}$ is the flow velocity, $\rho$ is the constant flow density, $P$ is the pressure, $\mu$ is the constant dynamic viscosity and $g$ is the sum of body forces. In what follows, the notations $U^1$, $U^0$ and $U^2$ designate the $x$, $y$ and $z$ components of the velocity respectively.

\subsection{Mathematical approach}
Each variable in the incompressible Navier-Stokes equations are expressed as four-variable power series. For instance, the x-velocity component becomes
\begin{equation}
U^1=\sum_{\omega,p,q,r=0}^\infty u^1_{\omega,p,q,r}t^\omega x^py^qz^r
\end{equation}
with the convention $\sum\limits_{\omega,p,q,r=0}^\infty=\sum\limits_{\omega=0}^\infty\sum\limits_{p=0}^\infty\sum\limits_{q=0}^\infty\sum\limits_{r=0}^\infty$.
It applies to $U^0$, $U^2$ and $P$ alike.
The domain of convergence of the x-velocity component corresponds to 
\begin{equation}
D_{U^1}=\{(t,x,y,z)\in \mathbb{R}^4:\sum\limits_{\omega,p,q,r=0}^\infty u^1_{\omega,p,q,r}t^\omega x^py^qz^r < \infty\}
\end{equation}
$D_{U^0}$, $D_{U^2}$ and $D_P$ are similarly defined. The domains of convergence of the power series are assumed to be non-empty. \newline
The resolution requires the body accelerations to be of the form
\begin{equation}
g^i=\sum_{\omega,p,q,r=0}^\infty g^i_{\omega,p,q,r}t^\omega x^py^qz^r
\end{equation}
with $i=0,1\text{ or }2$ corresponding to the directions (such as the velocity). This assumption is not constraining because it is generally true. For instance, gravity is often considered constant. \newline
The initial conditions allow the computation of the coefficients when $\omega=0$. When at least one of $p$, $q$ and $r$ is equal to $0$ or $1$, the velocity coefficients  are computed by the six velocity boundary conditions. As regards pressure, only the terms corresponding to $p=0$ or $q=0$ or $r=0$ are derived from the boundary conditions.

\subsection{Analytical solution}
This part itemizes the calculation steps leading to the analytical solution of the Navier-Stokes equations. 
\paragraph{\textbf{Step 1: Establishment of the equations}}
~~\\
The introduction of the power series in the continuity equation yields
\begin{equation}
\sum_{\omega,p,q,r=0}^\infty \left[(p+1)u^1_{\omega,p+1,q,r}+(q+1)u^0_{\omega,p,q+1,r}+(r+1)u^2_{\omega,p,q,r+1}\right] t^{\omega}x^py^qz^r=0
\label{Div}
\end{equation}
The coefficients of the power series in the equation \eqref{Div} are necessarily nil so:
\begin{equation}
u^0_{\omega,p,q,r}=-\frac{p+1}{q}u^1_{\omega,p+1,q-1,r}-\frac{r+1}{q}u^2_{\omega,p,q-1,r+1}~~~~~~\forall (\omega,p,r)\in\mathbb{N}^3,\forall~q>0
\label{DisDiv}
\end{equation}
The boundary conditions provide the terms corresponding to $q=0$. 
\newline
In the three momentum equations, the non linear terms lead to products of series calculated by the Cauchy product:
\begin{equation}
\left(\sum_{i=0}^\infty a_{i}x^i\right)\left(\sum_{j=0}^\infty b_{j}x^j\right)=\sum_{k=0}^\infty\left( \sum_{l=0}^k a_{l}b_{k-l}\right)x^k
\end{equation}
Only one momentum equation is written since the two others are similar. For instance, the x-momentum equation becomes
\footnotesize
\begin{multline}
\frac{1}{p+1}\bigg\{(\omega+1)u^1_{\omega+1,p,q+1,r+1}+\sum\limits_{k=0}^{\omega}\sum\limits_{i=0}^{p}\sum\limits_{j=0}^{q+1}\sum\limits_{l=0}^{r+1}\left[(i+1)u^1_{\omega-k,p-i,q+1-j,r+1-l}u^1_{k,i+1,j,l}+\right.\\
\left.(j+1)u^0_{\omega-k,p-i,q+1-j,r+1-l}u^1_{k,i,j+1,l}+(l+1)u^2_{\omega-k,p-i,q+1-j,r+1-l}u^1_{k,i,j,l+1}\right]-\nu\left[(p+1)(p+2)u^1_{\omega,p+2,q+1,r+1}+\right.\\ \left.
(q+2)(q+3)u^1_{\omega,p,q+3,r+1}+(r+2)(r+3)u^1_{\omega,p,q+1,r+3}\right]-g^1_{\omega,p,q+1,r+1}\bigg\}=-\frac{1}{\rho}P_{\omega,p+1,q+1,r+1}
\label{pressure}
\end{multline}
\normalsize
with $\omega\geq 0$, $p\geq 0$, $q\geq -1$ and $r\geq -1$.  \newline
The pressure coefficients in the equation \eqref{pressure} can be substituted thanks to the momentum equation along the y axis. The pressure term in the z-momentum equation is treated alike. This provides two large equations \eqref{momentumu} and \eqref{momentumw} combining $u^1$, $u^0$ and $u^2$ coefficients. 
\footnotesize
\begin{multline}
u^1_{\omega+1,p,q,r}=-\frac{p+1}{q(q-1)}\left[(p+2)u^1_{\omega+1,p+2,q-2,r}+(r+1)u^2_{\omega+1,p+1,q-2,r+1}\right]+\frac{\nu}{\omega+1}\left[\frac{(p+2)!}{p!}u^1_{\omega,p+2,q,r}+\frac{(q+2)!}{q!}\right. \\
\left.u^1_{\omega,p,q+2,r}+\frac{(r+2)!}{r!}u^1_{\omega,p,q,r+2}-\frac{(p+3)!}{q~p!}u^0_{\omega,p+3,q-1,r}-(p+1)(q+1)u^0_{\omega,p+1,q+1,r}-\frac{(p+1)(r+2)!}{q~r!}u^0_{\omega,p+1,q-1,r+2}\right]\\
-\frac{1}{\omega+1}\sum_{k=0}^{\omega}\bigg\{\sum_{i=0}^{p}\sum_{j=0}^{q}\sum_{l=0}^{r}\left[(i+1)u^1_{\omega-k,p-i,q-j,r-l}u^1_{k,i+1,j,l}+(j+1)u^0_{\omega-k,p-i,q-j,r-l}u^1_{k,i,j+1,l}+\right.\\
\left.(l+1)u^2_{\omega-k,p-i,q-j,r-l}u^1_{k,i,j,l+1}\right]-
\frac{p+1}{q} \sum_{i=0}^{p+1}\sum_{j=0}^{q-1}\sum_{l=0}^{r}\left[(i+1)u^1_{\omega-k,p+1-i,q-1-j,r-l}u^0_{k,i+1,j,l}+(j+1)\right.\\
\left. u^0_{\omega-k,p+1-i,q-1-j,r-l}u^0_{k,i,j+1,l}+(l+1)u^2_{\omega-k,p+1-i,q-1-j,r-l}u^0_{k,i,j,l+1}\right]\bigg\} 
+\frac{1}{\omega+1}\left(g^1_{\omega,p,q,r}-\frac{p+1}{q}g^0_{\omega,p+1,q-1,r}\right)
\label{momentumu}
\end{multline}

\begin{multline}
u^2_{\omega+1,p,q,r}=-\frac{r+1}{q(q-1)}\left[(r+2)u^2_{\omega+1,p,q-2,r+2}+(p+1)u^1_{\omega+1,p+1,q-2,r+1}\right]+\frac{\nu}{\omega+1}\left[\frac{(r+2)!}{r!}u^2_{\omega,p,q,r+2}+\frac{(q+2)!}{q!}\right. \\
\left.u^2_{\omega,p,q+2,r}+\frac{(p+2)!}{p!}u^2_{\omega,p+2,q,r}-\frac{(r+3)!}{q~r!}u^0_{\omega,p,q-1,r+3}-(r+1)(q+1)u^0_{\omega,p,q+1,r+1}-\frac{(r+1)(p+2)!}{q~p!}u^0_{\omega,p+2,q-1,r+1}\right]\\
-\frac{1}{\omega+1}\sum_{k=0}^{\omega}\bigg\{\sum_{i=0}^{p}\sum_{j=0}^{q}\sum_{l=0}^{r}\left[(i+1)u^1_{\omega-k,p-i,q-j,r-l}u^2_{k,i+1,j,l}+(j+1)u^0_{\omega-k,p-i,q-j,r-l}u^2_{k,i,j+1,l}+\right.\\
\left.(l+1)u^2_{\omega-k,p-i,q-j,r-l}u^2_{k,i,j,l+1}\right]-
\frac{r+1}{q} \sum_{i=0}^{p}\sum_{j=0}^{q-1}\sum_{l=0}^{r+1}\left[(i+1)u^1_{\omega-k,p-i,q-1-j,r+1-l}u^0_{k,i+1,j,l}+(j+1) \right.\\
\left.u^0_{\omega-k,p-i,q-1-j,r+1-l}u^0_{k,i,j+1,l}+(l+1)u^2_{\omega-k,p-i,q-1-j,r+1-l}u^0_{k,i,j,l+1}\right]\bigg\} 
+\frac{1}{\omega+1}\left(g^2_{\omega,p,q,r}-\frac{r+1}{q}g^0_{\omega,p,q-1,r+1}\right)
\label{momentumw}
\end{multline}
\normalsize
The equation \eqref{DisDiv} constitutes the third relation between the velocity coefficients. These three recurrence relations enable the calculation of the velocity. Then, the pressure coefficients are calculated from the velocities coefficients by a momentum equation such as \eqref{pressure}.\newline
Both prior equations \eqref{momentumu} and \eqref{momentumw} can be combined into one 
\footnotesize
\begin{multline}
u^\alpha_{\omega,p,q,r}=-\frac{X_\alpha+1}{q(q-1)}\left[(X_\alpha+2)u^\alpha_{\omega,p+2(-1)^{\alpha+1}+2(\alpha-1),q-2,r+2(-1)^{\alpha}+2(2-\alpha)}+(X_{3-\alpha}+1)\right.\\
\left.u^{3-\alpha}_{\omega,p+(-1)^{\alpha+1}+2(\alpha-1),q-2,r+(-1)^{\alpha}+2(2-\alpha)}\right]+\frac{\nu}{\omega}\left[\frac{(X_\alpha+2)!}{X_\alpha!}u^\alpha_{\omega-1,p+2(-1)^{\alpha+1}+2(\alpha-1),q,r+2(-1)^{\alpha}+2(2-\alpha)}+\frac{(q+2)!}{q!} \right. \\
\left.u^\alpha_{\omega-1,p,q+2,r}+\frac{(X_{3-\alpha}+2)!}{X_{3-\alpha}!}u^\alpha_{\omega-1,p+2(\alpha-1),q,r+2(2-\alpha)} -\frac{(X_\alpha+3)!}{q~X_\alpha!}u^0_{\omega-1,p+3(2-\alpha),q-1,r+3(\alpha-1)}-(X_\alpha+1)(q+1)\right.\\
\left.u^0_{\omega-1,p+2-\alpha,q+1,r+\alpha-1}-\frac{(X_\alpha+1)(X_{3-\alpha}+2)!}{q~X_{3-\alpha}!}u^0_{\omega-1,p+(-1)^{\alpha+1}+3\alpha-4+2(\alpha-1),q-1,r+2(-1)^\alpha-3\alpha+5+2(2-\alpha)}\right]+\\
\frac{1}{\omega}\sum_{k=0}^{\omega-1}\sum_{b=0}^{1}\sum_{\phi=0}^{2}\sum_{i=0}^{p+(1-b)(2-\alpha)}\sum_{j=0}^{q+b-1}\sum_{l=0}^{r+(1-b)(\alpha-1)}\bigg[\frac{1}{2}(\phi-1)(\phi-2)j+\phi(2-\phi)i+\frac{1}{2}\phi(\phi-1)l+1](-1)^b\frac{(b-1+q)!}{q!}\\
\frac{(X_\alpha+1-b)!}{X_\alpha!}u^\phi_{\omega-1-k,p+(1-b)(2-\alpha)-i,q-1+b-j,r+(1-b)(\alpha-1)-l}u^{b\alpha}_{k,i+\phi(2-\phi),j+\frac{1}{2}(\phi-1)(\phi-2),l+\frac{1}{2}\phi(\phi-1)}\bigg]\\
+\frac{1}{\omega}\left(g^\alpha_{\omega-1,p,q,r}-\frac{X_\alpha+1}{q}g^0_{\omega-1,p+2-\alpha,q-1,r+\alpha-1}\right)
\label{momentuma}
\end{multline}
\normalsize
The notation $X_i$ with $X_0=q$, $X_1=p$ and $X_2=r$ is adopted. The equation \eqref{momentuma} corresponds to the equation \eqref{momentumu} when $\alpha=1$ and to the equation \eqref{momentumw} when $\alpha=2$.
The problem now consists of solving \eqref{momentuma} and \eqref{DisDiv} and expressing the coefficients $u^\alpha_{\omega,p,q,r}$ as functions of the boundary coefficients $u^\alpha_{\omega,p=0\text{ or }1,q=0\text{ or }1,r=0\text{ or }1}$ and the initial coefficients $u^\alpha_{\omega=0,p,q,r}$. 
\paragraph{\textbf{Step 2: First compaction of the equations}}
~~\\
For the sake of convenience, the equation \eqref{momentuma} need to be compacted. The following proposition constitutes the first stage of the compaction. 
\begin{prop}
$\forall \omega\in\mathbb{N}^*,~\forall (p,q,r)\in\mathbb{N}_{>1}\times\mathbb{N}_{>1}\times\mathbb{N}_{>1},~\alpha=1\text{ or }2$, the coefficients $u^\alpha_{\omega,p,q,r}$ can be expressed as
\begin{multline}
u^\alpha_{\omega,p,q,r}=F(0,\omega,p,q,r,\alpha,0)+\sum_{z=1}^8 \sum_{n=1}^{E\left(\frac{q}{2}\right)} F(z,\omega,p,q,r,\alpha,n) u^{\kappa(z,\alpha)}_{\omega-1,\xi(z,p,q,\alpha,n),\Delta(z,q),\epsilon(z,q,r,\alpha,n)}+\\
\frac{1}{\omega}\sum_{k=0}^{\omega-1}\sum_{a=0}^2\sum_{b=0}^1\sum_{\phi=0}^2\sum_{n=1}^{E\left(\frac{q}{2}\right)-a}\sum_{m=1}^n\sum_{i=0}^{s_1(p,q,\alpha,a,n,m,b)}\sum_{j=0}^{s_2(q,a,n,b)}\sum_{l=0}^{s_3(q,r,\alpha,a,n,m,b)}\widetilde{G}(p,q,r,\alpha,a,n,m,b,\phi,i,j,l)\\
u^{b(\alpha+(a-2)(2\alpha-3)a)}_{k,i+\phi(2-\phi),j+\frac{1}{2}(\phi-1)(\phi-2),l+\frac{1}{2}\phi(\phi-1)}u^{\phi}_{\omega-1-k,Y_1(p,q,\alpha,a,n,m,b)-i,Y_0(q,a,n,b)-j,Y_2(q,r,\alpha,a,n,m,b)-l}
\label{Step1}
\end{multline}
with
\small
\begin{multline}
\widetilde{G}(p,q,r,\alpha,a,n,m,b,\phi,i,j,l)=H(q,a,n,m,b)\left[\frac{1}{2}(\phi-1)(\phi-2)j+\phi(2-\phi)i+\frac{1}{2}\phi(\phi-1)l+1\right]\\
\frac{Y_0(q,a,n,b)!Y_1(p,q,\alpha,a,n,m,b)!Y_2(q,r,\alpha,a,n,m,b)!}{p!q!r!}
\label{GG}
\end{multline}
\end{prop}
\normalsize
All the functions in the equations \eqref{Step1} and \eqref{GG} are given in the appendix A. The function $E$ corresponds to the floor function.
\begin{proof}
This formula is demonstrated by induction on $q$. 
\begin{itemize}
\item Base case: $q=2$ and $3$ 
\end{itemize}
In this case, the equation \eqref{momentuma} becomes
\footnotesize
\begin{multline}
u^\alpha_{\omega,p,q=2\text{ or }3,r}=F_0(\omega,p,q,r,\alpha)+[F(7,\omega,p,q,r,\alpha,1)u^{\kappa(7,\alpha)}_{\omega-1,\xi(7,p,q,\alpha,1),\Delta(7,q),\epsilon(7,q,r,\alpha,1)}+F(2,\omega,p,q,r,\alpha,1)\\
u^{\kappa(2,\alpha)}_{\omega-1,\xi(2,p,q,\alpha,1),\Delta(2,q),\epsilon(2,q,r,\alpha,1)}+
F(1,\omega,p,q,r,\alpha,1)u^{\kappa(1,\alpha)}_{\omega-1,\xi(1,p,q,\alpha,1),\Delta(1,q),\epsilon(1,q,r,\alpha,1)}+F(5,\omega,p,q,r,\alpha,1)\\u^{\kappa(5,\alpha)}_{\omega-1,\xi(5,p,q,\alpha,1),\Delta(5,q),\epsilon(5,q,r,\alpha,1)}+F(3,\omega,p,q,r,\alpha,1)u^{\kappa(3,\alpha)}_{\omega-1,\xi(3,p,q,\alpha,1),\Delta(3,q),\epsilon(3,q,r,\alpha,1)}+F(6,\omega,p,q,r,\alpha,1)\\
u^{\kappa(6,\alpha)}_{\omega-1,\xi(6,p,q,\alpha,1),\Delta(6,q),\epsilon(6,q,r,\alpha,1)}]+ 
\frac{1}{\omega}\sum_{k=0}^{\omega-1}\sum_{b=0}^{1}\sum_{\phi=0}^{2}\sum_{i=0}^{s_1(p,q,\alpha,0,1,1,b)}\sum_{j=0}^{s_2(q,0,1,b)}\sum_{l=0}^{s_3(q,r,\alpha,0,1,1,b)}\bigg[\widetilde{G}(p,q,r,\alpha,0,1,1,b,\phi,i,j,l)\\
u^{b\alpha}_{k,i+\phi(2-\phi),j+\frac{1}{2}(\phi-1)(\phi-2),l+\frac{1}{2}\phi(\phi-1)}u^{\phi}_{\omega-1-k,Y_1(p,q,\alpha,0,1,1,b)-i,Y_0(q,0,1,b)-j,Y_2(q,r,\alpha,0,1,1,b)-l}\bigg]
+\widetilde{F_0}(\omega,p,q,r,\alpha)
\label{II}
\end{multline}
\normalsize
With the following relations
\begin{equation}
F(0,\omega,p,q,r,\alpha,0)=F_0(\omega,p,q,r,\alpha)+\widetilde{F_0}(\omega,p,q,r,\alpha)
\end{equation}
\begin{equation}
F(4,\omega,p,q<4,r,\alpha,1)=F(8,\omega,p,q<4,r,\alpha,1)=0
\end{equation}
the equation \eqref{II} becomes
\begin{multline}
u^\alpha_{\omega,p,q=2\text{ or }3,r}=F(0,\omega,p,q,r,\alpha,0)+\sum_{z=1}^8 F(z,\omega,p,q,r,\alpha,n) u^{\kappa(z,\alpha)}_{\omega-1,\xi(z,p,q,\alpha,n),\Delta(z,q),\epsilon(z,q,r,\alpha,n)}+ \\
\frac{1}{\omega}\sum_{k=0}^{\omega-1}\sum_{b=0}^{1}\sum_{\phi=0}^{2}\sum_{i=0}^{s_1(p,q,\alpha,0,1,1,b)}\sum_{j=0}^{s_2(q,0,1,b)}\sum_{l=0}^{s_3(q,r,\alpha,0,1,1,b)}\bigg[\widetilde{G}(p,q,r,\alpha,0,1,1,b,\phi,i,j,l)\\
u^{b\alpha}_{k,i+\phi(2-\phi),j+\frac{1}{2}(\phi-1)(\phi-2),l+\frac{1}{2}\phi(\phi-1)}u^{\phi}_{\omega-1-k,Y_1(p,q,\alpha,0,1,1,b)-i,Y_0(q,0,1,b)-j,Y_2(q,r,\alpha,0,1,1,b)-l}\bigg]
\label{proofq2}
\end{multline}
The equation \eqref{proofq2} corresponds to \eqref{Step1} for $q=2$ or $3$.
\begin{itemize}
\item Induction step 
\end{itemize}
This step assumes that the formula \eqref{Step1} holds for any $q-2$. Considering the base case, it is sufficient to prove \eqref{Step1} for $q$. In the equation \eqref{momentuma2}, the two bold terms are replaced thanks to the induction hypothesis.
\footnotesize
\begin{multline}
u^\alpha_{\omega,p,q,r}=-\frac{X_\alpha+1}{q(q-1)}\left[(X_\alpha+2)\boldsymbol{u^\alpha_{\omega,p+2(-1)^{\alpha+1}+2(\alpha-1),q-2,r+2(-1)^{\alpha}+2(2-\alpha)}}+(X_{3-\alpha}+1)\right.\\
\left.\boldsymbol{u^{3-\alpha}_{\omega,p+(-1)^{\alpha+1}+2(\alpha-1),q-2,r+(-1)^{\alpha}+2(2-\alpha)}}\right]+\frac{\nu}{\omega}\left[\frac{(X_\alpha+2)!}{X_\alpha!}u^\alpha_{\omega-1,p+2(-1)^{\alpha+1}+2(\alpha-1),q,r+2(-1)^{\alpha}+2(2-\alpha)}+\right. \\
\left.\frac{(q+2)!}{q!}u^\alpha_{\omega-1,p,q+2,r}+\frac{(X_{3-\alpha}+2)!}{X_{3-\alpha}!}u^\alpha_{\omega-1,p+2(\alpha-1),q,r+2(2-\alpha)}-\frac{(X_\alpha+3)!}{q~X_\alpha!}u^0_{\omega-1,p+3(2-\alpha),q-1,r+3(\alpha-1)} -(X_\alpha+1)\right.\\
\left.(q+1)u^0_{\omega-1,p+2-\alpha,q+1,r+\alpha-1}-\frac{(X_\alpha+1)(X_{3-\alpha}+2)!}{q~X_{3-\alpha}!}u^0_{\omega-1,p+(-1)^{\alpha+1}+3\alpha-4+2(\alpha-1),q-1,r+2(-1)^\alpha-3\alpha+5+2(2-\alpha)}\right]
+\\
\frac{1}{\omega}\sum_{k=0}^{\omega-1}\sum_{b=0}^{1}\sum_{\phi=0}^{2}\sum_{i=0}^{p+(1-b)(2-\alpha)}\sum_{j=0}^{q+b-1}\sum_{l=0}^{r+(1-b)(\alpha-1)}\bigg\{\bigg[\frac{1}{2}(\phi-1)(\phi-2)j+\phi(2-\phi)i+\frac{1}{2}\phi(\phi-1)l+\bigg](-1)^b\frac{(b-1+q)!}{q!}\\
\frac{(X_\alpha+1-b)!}{X_\alpha!}u^\phi_{\omega-1-k,p+(1-b)(2-\alpha)-i,q-1+b-j,r+(1-b)(\alpha-1)-l}u^{b\alpha}_{k,i+\phi(2-\phi),j+\frac{1}{2}(\phi-1)(\phi-2),l+\frac{1}{2}\phi(\phi-1)}\bigg\}
+\\
\frac{1}{\omega}\left(g^\alpha_{\omega-1,p,q,r}-\frac{X_\alpha+1}{q}g^0_{\omega-1,p+2-\alpha,q-1,r+\alpha-1}\right)
\label{momentuma2}
\end{multline}
\normalsize
Let $\mathcal{X}$ be the following function
\footnotesize
\begin{multline}
\mathcal{X}(\omega,p,q,r,\alpha)=\sum_{k=0}^{\omega-1}\sum_{a=0}^2\sum_{b=0}^1\sum_{\phi=0}^2\sum_{n=1}^{E\left(\frac{q}{2}\right)-a}\sum_{m=1}^n\sum_{i=0}^{s_1(p,q,\alpha,a,n,m,b)}\sum_{j=0}^{s_2(q,a,n,b)}\sum_{l=0}^{s_3(q,r,\alpha,a,n,m,b)}\widetilde{G}(p,q,r,\alpha,a,n,m,b,\phi,i,j,l)\\
u^{b(\alpha+(a-2)(2\alpha-3)a)}_{k,i+\phi(2-\phi),j+\frac{1}{2}(\phi-1)(\phi-2),l+\frac{1}{2}\phi(\phi-1)}u^{\phi}_{\omega-1-k,Y_1(p,q,\alpha,a,n,m,b)-i,Y_0(q,a,n,b)-j,Y_2(q,r,\alpha,a,n,m,b)-l}
\end{multline}
\normalsize
Then, the two bold terms in the equation \eqref{momentuma2} are
\footnotesize
\begin{multline}
u^\alpha_{\omega,p+2(-1)^{\alpha+1}+2(\alpha-1),q-2,r+2(-1)^{\alpha}+2(2-\alpha)}=F(0,\omega,p+2(-1)^{\alpha+1}+2(\alpha-1),q-2,r+2(-1)^{\alpha}+2(2-\alpha),\alpha,0)\\
+\sum_{z=1}^8 \sum_{n=1}^{E\left(\frac{q}{2}\right)-1} \bigg\{F(z,\omega,p+2(-1)^{\alpha+1}+2(\alpha-1),q-2,r+2(-1)^{\alpha}+2(2-\alpha),\alpha,n)\\
u^{\kappa(z,\alpha)}_{\omega-1,\xi(z,p+2(-1)^{\alpha+1}+2(\alpha-1),q-2,\alpha,n),\Delta(z,q-2),\epsilon(z,q-2,r+2(-1)^{\alpha}+2(2-\alpha),\alpha,n)}\bigg\}+\\
\frac{1}{\omega}\mathcal{X}(\omega,p+2(-1)^{\alpha+1}+2(\alpha-1),q-2,r+2(-1)^{\alpha}+2(2-\alpha),\alpha)
\label{Form1}
\end{multline}
\normalsize
and
\footnotesize
\begin{multline}
u^{3-\alpha}_{\omega,p+(-1)^{\alpha+1}+2(\alpha-1),q-2,r+(-1)^{\alpha}+2(2-\alpha)}= F(0,\omega,p+(-1)^{\alpha+1}+2(\alpha-1),q-2,r+(-1)^{\alpha}+2(2-\alpha),3-\alpha,0)\\
+\sum_{z=1}^8 \sum_{n=1}^{E\left(\frac{q}{2}\right)-1}\bigg\{ F(z,\omega,p+(-1)^{\alpha+1}+2(\alpha-1),q-2,r+(-1)^{\alpha}+2(2-\alpha),3-\alpha,n) \\
u^{\kappa(z,3-\alpha)}_{\omega-1,\xi(z,p+(-1)^{\alpha+1}+2(\alpha-1),q-2,\alpha,n),\Delta(z,q-2),\epsilon(z,q-2,r+(-1)^{\alpha}+2(2-\alpha),3-\alpha,n)}\bigg\}+\\
\frac{1}{\omega}\mathcal{X}(\omega,p+(-1)^{\alpha+1}+2(\alpha-1),q-2,r+(-1)^{\alpha}+2(2-\alpha),3-\alpha)
\label{Form2}
\end{multline}
\normalsize
Formal calculations demonstrate the four following relations.
\begin{multline}
F_0(\omega,p,q,r,\alpha)=-\frac{X_{\alpha}+1}{q(q-1)}\bigg[(X_{\alpha}+2)F_0(\omega,p+2(-1)^{\alpha+1}+2(\alpha-1),q-2,r+2(-1)^{\alpha}+\\2(2-\alpha),\alpha)+(X_{3-\alpha}+1)F_0(\omega,p+(-1)^{\alpha+1}+2(\alpha-1),q-2,r+(-1)^{\alpha}+2(2-\alpha),3-\alpha)\bigg]
\label{eq1}
\end{multline}
\begin{multline}
\widetilde{F_0}(\omega,p,q,r,\alpha)=-\frac{X_{\alpha}+1}{q(q-1)}\bigg[(X_{\alpha}+2)\widetilde{F_0}(\omega,p+2(-1)^{\alpha+1}+2(\alpha-1),q-2,r+2(-1)^{\alpha}+\\2(2-\alpha),\alpha)+(X_{3-\alpha}+1)\widetilde{F_0}(\omega,p+(-1)^{\alpha+1}+2(\alpha-1),q-2,r+(-1)^{\alpha}+2(2-\alpha),3-\alpha)\bigg]
\label{eq2}
\end{multline}
\footnotesize
\begin{multline}
\sum_{z=1}^8 \sum_{n=1}^{E\left(\frac{q}{2}\right)} F(z,\omega,p,q,r,\alpha,n) u^{\kappa(z,\alpha)}_{\omega-1,\xi(z,p,q,\alpha,n),\Delta(z,q),\epsilon(z,q,r,\alpha,n)}=-\frac{X_\alpha+1}{q(q-1)}\sum_{z=1}^8 \sum_{n=1}^{E\left(\frac{q}{2}\right)-1} \\
\bigg[(X_\alpha+2) F(z,\omega,p+2(-1)^{\alpha+1}+2(\alpha-1),q-2,r+2(-1)^\alpha+2(2-\alpha),\alpha,n) \\
u^{\kappa(z,\alpha)}_{\omega-1,\xi(z,p+2(-1)^{\alpha+1}+2(\alpha-1),q-2,\alpha,n),\Delta(z,q-2),\epsilon(z,q-2,r+2(-1)^\alpha+2(2-\alpha),\alpha,n)}+\\
(X_{3-\alpha}+1)F(z,\omega,p+(-1)^{\alpha+1}+2(\alpha-1),q-2,r+(-1)^\alpha+2(2-\alpha),3-\alpha,n)\\
 u^{\kappa(z,3-\alpha)}_{\omega-1,\xi(z,p+(-1)^{\alpha+1}+2(\alpha-1),q-2,3-\alpha,n),\Delta(z,q-2),\epsilon(z,q-2,r+(-1)^\alpha+2(2-\alpha),3-\alpha,n)}\bigg]+\frac{\nu}{\omega}\bigg[\frac{(X_\alpha+2)!}{X_\alpha!}\\
u^\alpha_{\omega-1,p+2(-1)^{\alpha+1}+2(\alpha-1),q,r+2(-1)^{\alpha}+2(2-\alpha)}+\frac{(q+2)!}{q!}u^\alpha_{\omega-1,p,q+2,r}+\frac{(X_{3-\alpha}+2)!}{X_{3-\alpha}!}u^\alpha_{\omega-1,p+2(\alpha-1),q,r+2(2-\alpha)}-\\
\frac{(X_\alpha+3)!}{q~X_\alpha!}u^0_{\omega-1,p+3(2-\alpha),q-1,r+3(\alpha-1)} -(X_\alpha+1)(q+1)u^0_{\omega-1,p+2-\alpha,q+1,r+\alpha-1}-\\
\frac{(X_\alpha+1)(X_{3-\alpha}+2)!}{q~X_{3-\alpha}!}u^0_{\omega-1,p+(-1)^{\alpha+1}+3\alpha-4+2(\alpha-1),q-1,r+2(-1)^\alpha-3\alpha+5+2(2-\alpha)}\bigg]
\label{eq3}
\end{multline}
\begin{multline}
\mathcal{X}(\omega,p,q,r,\alpha)=-\frac{X_\alpha+1}{q(q-1)}\bigg[(X_\alpha+2)\mathcal{X}(\omega,p+2(-1)^{\alpha+1}+2(\alpha-1),q-2,r+2(-1)^{\alpha}+2(2-\alpha),\alpha)+\\
(X_{3-\alpha}+1)\mathcal{X}(\omega,p+(-1)^{\alpha+1}+2(\alpha-1),q-2,r+(-1)^{\alpha}+2(2-\alpha),3-\alpha)\bigg]+\frac{1}{\omega}\sum_{k=0}^{\omega-1}\sum_{b=0}^{1}\sum_{\phi=0}^{2}\sum_{i=0}^{p+(1-b)(2-\alpha)}\\
\sum_{j=0}^{q+b-1}\sum_{l=0}^{r+(1-b)(\alpha-1)}[\frac{1}{2}(\phi-1)(\phi-2)j+\phi(2-\phi)i+\frac{1}{2}\phi(\phi-1)l+1](-1)^b\frac{(b-1+q)!}{q!}\frac{(X_\alpha+1-b)!}{X_\alpha!}\\
u^\phi_{\omega-1-k,p+(1-b)(2-\alpha)-i,q-1+b-j,r+(1-b)(\alpha-1)-l}u^{b\alpha}_{k,i+\phi(2-\phi),j+\frac{1}{2}(\phi-1)(\phi-2),l+\frac{1}{2}\phi(\phi-1)}
\label{eq4}
\end{multline}
\normalsize
The substitution of the equations \eqref{Form1} and \eqref{Form2} in the equation \eqref{momentuma2} and then the substitution of the equations from \eqref{eq1} to \eqref{eq4} in the resulting relation leads to the relation \eqref{Step1} for $q$. Then, the formula holds $\forall~q>1$ which completes the proof.
\end{proof}
\paragraph{\textbf{Step 3: Second compaction of the equations}}
~~\\
The formulation of the problem is still too complex and a second compaction is needed. Note that a convention will be used in the remaining: when at least one of $\omega$, $p$, $q$ and $r$ is negative, then $u^\alpha_{\omega,p,q,r}=1$.
\begin{prop}
The equations \eqref{Step1} and \eqref{DisDiv} are equivalent to \eqref{step3}
\begin{multline}
u^\alpha_{\omega,p,q,r}=\sum_{i=0}^{\mathcal{F}(p,q,r,\alpha,\omega)}\Gamma(\omega,p,q,r,\alpha,i)
u^{\beta(\omega,p,q,r,\alpha,i)}_{\mathcal{E}(\omega,p,q,r,\alpha,i),T_1(\omega,p,q,r,\alpha,i),T_0(\omega,p,q,r,\alpha,i),T_2(\omega,p,q,r,\alpha,i)}\\
u^{\Phi(\omega,p,q,r,\alpha,i)}_{\Upsilon(\omega,p,q,r,\alpha,i),Z_1(\omega,p,q,r,\alpha,i),Z_0(\omega,p,q,r,\alpha,i),Z_2(\omega,p,q,r,\alpha,i)}
\label{step3}
\end{multline}
\end{prop}

\begin{proof}
This proposition corresponds to the merger of multiple sums into one. 
The first double sums in the equation \eqref{Step1} can be replaced by one sum 
\begin{equation}
\sum_{z=1}^8 \sum_{n=1}^{E\left(\frac{q}{2}\right)} \Leftrightarrow \sum_{\chi=1}^{8E\left(\frac{q}{2}\right)}
\end{equation}
with
\begin{equation}
n=E\left(\frac{\chi-1}{8}\right)+1
\end{equation}
\begin{equation}
z=\chi-8E_+\left(\frac{\chi-1}{8}\right)
\end{equation}
$E_+$ is the floor function which is equal to 0 when the argument is negative. Moreover, the first term in the equation \eqref{Step1} can integrate the previous sum:
\begin{multline}
F(0,\omega,p,q,r,\alpha,0)+\sum_{\chi=1}^{8E\left(\frac{q}{2}\right)} F(z,\omega,p,q,r,\alpha,n) u^{\kappa(z,\alpha)}_{\omega-1,\xi(z,p,q,\alpha,n),\Delta(z,q),\epsilon(z,q,r,\alpha,n)} =\\ \sum_{\chi=0}^{8E\left(\frac{q}{2}\right)} F(z,\omega,p,q,r,\alpha,n) u^{\kappa(z,\alpha)}_{\omega-1,\xi(z,p,q,\alpha,n),\Delta(z,q),\epsilon(z,q,r,\alpha,n)}
\end{multline}
Only one term is added in the sum, $\chi=0$. In this case, $z=0$ and the functions $\xi$, $\Delta$ and $\epsilon$ are equal to $-1$ by definition. Then, the coefficient $u$ in the sum is equal to $1$ by convention. As expected, the first term of the sum corresponds to $F(0,\omega,p,q,r,\alpha,0)$. \newline
Regarding the second part of the equation \eqref{Step1}, the multiple sums are divided into two parts. First,
\begin{equation}
\sum_{i=0}^{s_1(p,q,\alpha,a,n,m,b)}\sum_{j=0}^{s_2(q,a,n,b)}\sum_{l=0}^{s_3(q,r,\alpha,a,n,m,b)} \Leftrightarrow \sum_{k=1}^{(s_1(p,q,\alpha,a,n,m,b)+1)(s_2(q,a,n,b)+1)(s_3(q,r,\alpha,a,n,m,b)+1)}
\end{equation}
with 
\begin{equation}
\left\lbrace
\begin{array}{lc}
i=E\left(\frac{k-1}{(s_3(q,r,\alpha,a,n,m,b)+1)(s_2(q,a,n,b)+1)}\right) \\
j=E\left(\frac{k-1}{(s_3(q,r,\alpha,a,n,m,b)+1)}\right)-(s_2(q,a,n,b)+1)E\left(\frac{k-1}{(s_3(q,r,\alpha,a,n,m,b)+1)(s_2(q,a,n,b)+1)}\right) \\
l=\mathcal{M}od(k-1,s_3(q,r,\alpha,a,n,m,b)+1)
\end{array}\right.
\end{equation}
The function $\mathcal{M}od$ corresponds to the modulo operation. Second,
\begin{equation}
\sum_{a=0}^2\sum_{b=0}^1\sum_{\phi=0}^2\sum_{n=1}^{E\left(\frac{q}{2}\right)-a}\sum_{m=1}^n \Leftrightarrow \sum_{d=0}^{9E\left(\frac{q}{2}\right)^2-9E\left(\frac{q}{2}\right)+5}
\end{equation}
with
\begin{equation}
\left\lbrace
\begin{array}{lc}
b=B(d,q) \\
a=A(d,q) \\
\phi=\varphi(d,q) \\
n=N(d,q) \\
m=M(d,q)
\end{array}\right.
\end{equation}
The two prior sums can be changed into
\begin{equation}
\sum_{d=0}^{9E\left(\frac{q}{2}\right)^2-9E\left(\frac{q}{2}\right)+5}\sum_{i=1}^{(s_1(p,q,\alpha,a,n,m,b)+1)(s_2(q,a,n,b)+1)(s_3(q,r,\alpha,a,n,m,b)+1)} \Leftrightarrow \sum_{\tau=1}^{\sigma(p,q,r,\alpha)}
\end{equation}
with 
\begin{align}
\begin{array}{lc}
d=D(p,q,r,\alpha,\tau)\\
i=\mathcal{I}(p,q,r,\alpha,\tau)
\end{array}
\end{align}
and
\small
\begin{equation}
\sigma(p,q,r,\alpha)=\sum\limits_{d=0}^{9E\left(\frac{q}{2}\right)^2-9E\left(\frac{q}{2}\right)+5} (s_1(p,q,\alpha,a,n,m,b)+1)(s_2(q,a,n,b)+1)(s_3(q,r,\alpha,a,n,m,b)+1)
\end{equation}
\normalsize
Formal calculations provide a simpler formula for $\sigma$ given by \eqref{sigma}.\newline
Then, the equation \eqref{Step1} becomes
\begin{multline}
u^\alpha_{\omega,p,q,r}=\sum_{\chi=0}^{8E\left(\frac{q}{2}\right)}F(z,\omega,p,q,r,\alpha,n)
 u^{\kappa(z,\alpha)}_{\omega-1,\xi(z,p,q,\alpha,n),\Delta(z,q),\epsilon(z,q,r,\alpha,n)}+\\
\frac{1}{\omega}\sum_{k=0}^{\omega-1}\sum_{\tau=1}^{\sigma(p,q,r,\alpha)}G(p,q,r,\alpha,d,i)
u^{B(d,q)(\alpha+(A(d,q)-2)(2\alpha-3)A(d,q))}_{k,W_1(p,q,r,\alpha,d,i),W_0(p,q,r,\alpha,d,i),W_2(p,q,r,\alpha,d,i)}\\
u^{\varphi(d,q)}_{\omega-1-k,V_1(p,q,r,\alpha,d,i),V_0(q,r,\alpha,d,i),V_2(q,r,\alpha,d,i)}
\label{step2_5}
\end{multline}
The double sum corresponds to one sum
\begin{equation}
\sum_{k=0}^{\omega-1}\sum_{\tau=1}^{\sigma(p,q,r,\alpha)} \Leftrightarrow \sum_{j=0}^{\omega\sigma(p,q,r,\alpha)-1}
\end{equation}
with
\begin{equation}
\left\lbrace
\begin{array}{lc}
\tau=\mathcal{T}(p,q,r,\alpha,j) \\
k=E\left(\frac{j}{\sigma(p,q,r,\alpha)}\right)
\end{array}\right.
\end{equation}
Both remaining sums are merged into one sum $\sum\limits_{i=0}^{\mathcal{F}(p,q,r,\alpha,\omega)}$ where the terms from $i=0$ to $i=\omega\sigma(p,q,r,\alpha)-1$ correspond to the second sum and the terms from $i=\omega\sigma(p,q,r,\alpha)$ to $i=\omega\sigma(p,q,r,\alpha)+8E\left(\frac{q}{2}\right)$ correspond to the first sum. Then,
\begin{multline}
u^\alpha_{\omega,p,q,r}=\sum_{i=0}^{\mathcal{F}(p,q,r,\alpha,\omega)}\Gamma(\omega,p,q,r,\alpha,i)
u^{\beta(\omega,p,q,r,\alpha,i)}_{\mathcal{E}(\omega,p,q,r,\alpha,i),T_1(\omega,p,q,r,\alpha,i),T_0(\omega,p,q,r,\alpha,i),T_2(\omega,p,q,r,\alpha,i)}\\
u^{\Phi(\omega,p,q,r,\alpha,i)}_{\Upsilon(\omega,p,q,r,\alpha,i),Z_1(\omega,p,q,r,\alpha,i),Z_0(\omega,p,q,r,\alpha,i),Z_2(\omega,p,q,r,\alpha,i)}
\label{step33}
\end{multline}
The equation \eqref{step33} is applicable for $\alpha=1$ and $2$ but also for $\alpha=0$. In the case $\alpha=0$, the sum is twice bigger and all the functions employed in \eqref{step33} successively correspond to the case $\alpha=1$ and $\alpha=2$ according to the equation \eqref{DisDiv}. 
\end{proof}

\paragraph{\textbf{Step 4: The solution}}
~~\\
The resolution of the equation \eqref{step3} leads to the following theorem
\begin{thm}
The velocity coefficients are equal to
\footnotesize
\begin{multline}
u^\alpha_{\omega,p,q,r}=\sum\limits_{i=0}^{\eta(p,q,r,\alpha,\omega,\omega-1)}\bigg\{ \bigg[\prod\limits_{j=1}^{\omega}\prod\limits_{k=1}^{2^{j-1}}\Gamma(\rho_\omega(i,1,k,j,p,q,r,\alpha,\omega),\rho_1(i,1,k,j,p,q,r,\alpha,\omega),\rho_0(i,1,k,j,p,q,r,\alpha,\omega), \\
 \rho_2(i,1,k,j,p,q,r,\alpha,\omega),\rho_\alpha(i,1,k,j,p,q,r,\alpha,\omega),\Theta(k-2+2^{j-1},i,p,q,r,\alpha,\omega,0))\bigg] \\
\bigg[\prod\limits_{k=1}^{2^\omega}u^{\aleph_\alpha(i,1,k,p,q,r,\alpha,\omega)}_{\aleph_\omega(i,1,k,p,q,r,\alpha,\omega),\aleph_1(i,1,k,p,q,r,\alpha,\omega),\aleph_0(i,1,k,p,q,r,\alpha,\omega),\aleph_2(i,1,k,p,q,r,\alpha,\omega)}\bigg]\bigg\}
\label{ucoef}
\end{multline}
\normalsize
The pressure coefficients are calculated by the equation \eqref{pressure}. 
With these coefficients,
\begin{equation}
\left\lbrace
\begin{array}{lc}
U^\alpha=\sum\limits_{\omega,p,q,r=0}^\infty u^\alpha_{\omega,p,q,r}t^\omega x^py^qz^r\\
P=\sum\limits_{\omega,p,q,r=0}^\infty P_{\omega,p,q,r}t^\omega x^py^qz^r\\
\end{array}\right.
\end{equation}
satisfies the incompressible Navier-Stokes equations \eqref{NS}.
\end{thm}
All the functions required to calculate the coefficients are defined in the appendix A. Note that the functions $\aleph_x$ are such that the solution \eqref{ucoef} only depends on the boundary and initial conditions. The proof of the theorem is presented in the appendix B.\newline
The solution requires the calculation of recursive functions ($\rho_x$, $\Theta$ and $\aleph_x$) which demands excessive computational cost. No simple formulations for these functions were figured out and building abacuses seems to be the easiest way to reduce calculation time. These functions should contain the complexity of turbulence. 

\subsection{Domain of convergence}
The complex expression of $u^\alpha_{\omega,p,q,r}$ seems prohibitive to find the domain of convergence as a function of boundary and initial conditions. Therefore, the existence of an analytic solution cannot be maintained in the general case. \newline
However, a sufficient condition to have a domain of convergence of $[0,1[^4$ can be given. Let $max~\Gamma$ be
\begin{multline}
max~\Gamma=\displaystyle \max_{i,j,k}\mid\Gamma(\rho_\omega(i,1,k,j,p,q,r,\alpha,\omega),\rho_1(i,1,k,j,p,q,r,\alpha,\omega),\rho_0(i,1,k,j,p,q,r,\alpha,\omega), \\
 \rho_2(i,1,k,j,p,q,r,\alpha,\omega),\rho_\alpha(i,1,k,j,p,q,r,\alpha,\omega),\Theta(k-2+2^{j-1},i,p,q,r,\alpha,\omega,0))\mid
\end{multline}
and $max~u_\aleph$ be
\begin{multline}
max~u_\aleph=\displaystyle \max_{i,k}\mid u^{\aleph_\alpha(i,1,k,p,q,r,\alpha,\omega)}_{\aleph_\omega(i,1,k,p,q,r,\alpha,\omega),\aleph_1(i,1,k,p,q,r,\alpha,\omega),\aleph_0(i,1,k,p,q,r,\alpha,\omega),\aleph_2(i,1,k,p,q,r,\alpha,\omega)}\mid
\end{multline}
According to the equation \eqref{ucoef}, 
\begin{equation}
\mid u^\alpha_{\omega,p,q,r}\mid\leq\mid\eta(p,q,r,\alpha,\omega,\omega-1)+1\mid \left(max~\Gamma\right)^{2^\omega-1}\left(max~u_\aleph\right)^{2^\omega}
\end{equation}
Then, since $max~\Gamma>1$, the inequality \eqref{suff} is a sufficient condition to ensure a non-empty domain of convergence
\begin{equation}
max~u_\aleph\leq\frac{1}{max~\Gamma\mid\eta(p,q,r,\alpha,\omega,\omega-1)+1\mid^{1/2^\omega}}
\label{suff}
\end{equation}
In such a case, the domain of convergence at least contains $D=\{(t,x,y,z)\in ]-1,1[^4\}$. To go even further in this approach, the condition \eqref{suff2} assures the convergence on $\mathbb{R}^4$.
\begin{equation}
max~u_\aleph\leq\frac{1}{\omega!p!q!r!}\frac{1}{max~\Gamma\mid\eta(p,q,r,\alpha,\omega,\omega-1)+1\mid^{1/2^\omega}}
\label{suff2}
\end{equation}
In both cases, the solution deriving from \eqref{ucoef} takes part in the set of solutions. Moreover, the solution is analytic on a part or the whole space $\mathbb{R}^4$. Therefore, under some conditions such that \eqref{suff2}, the prior analysis proves the existence of a smooth solution to the Navier-Stokes equations. Note that the existence conditions above are rude and less constraining limitations could be sufficient. 
\section{Conclusion}
The investigation concerns the solution of the incompressible Navier-Stokes equations resulting from the application of the power series method. This approach demands the solution be a power series.\newline
Since this technique is unusual for multidimensional problems, an accessible differential equation is solved with the aim of displaying some challenges. In particular, the boundary conditions need to be of the form of power series.\newline
The power series method is then applied to the incompressible Navier-Stokes equations. All the steps leading to the analytical solution are itemized and proven. However, the complexity of the solution makes it difficult to estimate the domain of convergence. Then, the existence of a power series (smooth) solution to the Navier-Stokes equations is established only under specific conditions.\newline 
This study constitutes advances in fluid mechanics since no general analytical solution was established before. But to complete this analysis, the domain of convergence must be studied in a general context. Finer conditions depending on the boundary and initial conditions should exist. This achievement would contribute to a better understanding of the analytic solutions to the Navier-Stokes equations.

\newpage

\appendix
\section{Functions}
The appendix itemizes the functions necessary in the calculation of the analytical solution \eqref{ucoef}. Hereinafter, $E$ is the floor function, $E_+$ is the floor function which is equal to zero when the argument is negative, $\mathcal{H}$ is the Heaviside step function, $\delta$ is the Kronecker delta and $\mathcal{M}od(x,y)$ means $x$ modulo $y$. 
\begin{equation}
Ip(x)=
\left\lbrace
\begin{array}{lcc}
2&\text{ if $E\left(\frac{x}{2}\right)=\frac{x}{2}$}\\
\\
1&\text{ if $E\left(\frac{x}{2}\right)\neq\frac{x}{2}$}
\end{array}\right.
\end{equation}
\\
\begin{equation}
Ip_b(j,k)=
\left\lbrace
\begin{array}{lcc}
0&\text{ if $j=0$}\\
Ip\left(E\left(\frac{k+1-2^{j-1}}{2^{j-1}}\right)\right)&\text{ if $j\neq 0$}\\
\end{array}\right.
\end{equation}
\\
\begin{equation}
\mathcal{R}(x,y)=x-y E\left(\frac{x}{y}\right)
\label{R}
\end{equation}
\\
\begin{equation}
\xi(z,p,q,\alpha,n)=
\left\lbrace
\begin{array}{lcc}
-1&\text{ if }z=0 \\
p+2(\alpha-1)&\text{ if }z=1 \\
p&\text{ if }z=2 \\
p+2-\alpha&\text{ if }z=3 \\
p+1&\text{ if }z=4 \\
p+(2-\alpha)\left(2E\left(\frac{q}{2}\right)+1\right)&\text{ if }z=5 \\
p+(-1)^{(\alpha+1)}2n+3\alpha-4+2(\alpha-1)E\left(\frac{q}{2}\right)&\text{ if }z=6 \\
p+(-1)^{(\alpha+1)}2n+2(\alpha-1)E\left(\frac{q}{2}\right)&\text{ if }z=7 \\
p+(-1)^{(\alpha+1)}(2n-1)+2(\alpha-1)E\left(\frac{q}{2}\right)&\text{ if }z=8 \\
\end{array}\right.
\end{equation}
\\
\begin{equation}
\Delta(z,q)=
\left\lbrace
\begin{array}{lcc}
-1&\text{ if }z=0 \\
q&\text{ if }z=1 \\
q+2&\text{ if }z=2 \\
q+1&\text{ if }z=3 \\
q&\text{ if }z=4 \\
q-2E\left(\frac{q}{2}\right)+1&\text{ if }z=5 \\
q-2E\left(\frac{q}{2}\right)+1&\text{ if }z=6 \\
q-2E\left(\frac{q}{2}\right)+2&\text{ if }z=7 \\
q-2E\left(\frac{q}{2}\right)+2&\text{ if }z=8 \\
\end{array}\right.
\end{equation}
\\
\begin{equation}
\epsilon(z,q,r,\alpha,n)=
\left\lbrace
\begin{array}{lcc}
-1&\text{ if }z=0 \\
r+2(2-\alpha)&\text{ if }z=1 \\
r&\text{ if }z=2 \\
r+\alpha-1&\text{ if }z=3 \\
r+1&\text{ if }z=4 \\
r+(\alpha-1)\left(2E\left(\frac{q}{2}\right)+1\right)&\text{ if }z=5 \\
r+(-1)^{\alpha}2n-3\alpha+5+2(2-\alpha)E\left(\frac{q}{2}\right)&\text{ if }z=6 \\
r+(-1)^{\alpha}2n+2(2-\alpha)E\left(\frac{q}{2}\right)&\text{ if }z=7 \\
r+(-1)^{\alpha}(2n-1)+2(2-\alpha)E\left(\frac{q}{2}\right)&\text{ if }z=8 \\
\end{array}\right.
\end{equation}
\\
\begin{equation}
\kappa(z,\alpha)=
\left\lbrace
\begin{array}{lcc}
-1&\text{ if }z=0 \\
\alpha&\text{ if }z=1 \\
\alpha&\text{ if }z=2 \\
0&\text{ if }z=3 \\
3-\alpha&\text{ if }z=4 \\
0&\text{ if }z=5 \\
0&\text{ if }z=6 \\
\alpha&\text{ if }z=7 \\
3-\alpha&\text{ if }z=8 \\
\end{array}\right.
\end{equation}
\\
\begin{equation}
s_1(p,q,\alpha,a,n,m,b)=
\left\lbrace
\begin{array}{lcc}
p+(2-\alpha)\left(2E\left(\frac{q}{2}\right)+1-b-2n\right)&\text{ if }a=0 \\
p+(-1)^{\alpha}(2m-b)+(2-\alpha)(2n+1-b)&\text{ if }a=1 \\
p+2(-1)^{\alpha}m+(2-\alpha)(2n+3-b)&\text{ if }a=2 \\
\end{array}\right.
\end{equation}
\\
\begin{equation}
s_2(q,a,n,b)=
\left\lbrace
\begin{array}{lcc}
q-\left(2E\left(\frac{q}{2}\right)+1-b-2n\right)&\text{ if }a=0 \\
q-1-2n+b&\text{ if }a=1 \\
q-2n-3+b&\text{ if }a=2 \\
\end{array}\right.
\end{equation}
\\
\begin{equation}
s_3(q,r,\alpha,a,n,m,b)=
\left\lbrace
\begin{array}{lcc}
r+(\alpha-1)\left(2E\left(\frac{q}{2}\right)+1-b-2n\right)&\text{ if }a=0 \\
r+(-1)^{\alpha+1}(2m-b)+(\alpha-1)(2n+1-b)&\text{ if }a=1 \\
r+2(-1)^{\alpha+1}m+(\alpha-1)(2n+3-b)&\text{ if }a=2 \\
\end{array}\right.
\end{equation}
\\
\begin{equation}
B(d,q)=E\left(\frac{1}{4}\left(1+E\left(\frac{d}{\frac{3}{2}E\left(\frac{q}{2}\right)^2-\frac{3}{2}E\left(\frac{q}{2}\right)+1}\right)\right)\right)
\end{equation}
\\
\begin{equation}
A(d,q)=
\left\lbrace
\begin{array}{lcc}
0&\text{ if }\mathcal{M}od\left(d,\frac{3}{2}E\left(\frac{q}{2}\right)^2-\frac{3}{2}E\left(\frac{q}{2}\right)+1\right)<\frac{1}{2}E\left(\frac{q}{2}\right)\left(E\left(\frac{q}{2}\right)+1\right) \\
1&\text{ if } \frac{1}{2}E\left(\frac{q}{2}\right)\left(E\left(\frac{q}{2}\right)+1\right)-1<\mathcal{M}od\left(d,\frac{3}{2}E\left(\frac{q}{2}\right)^2-\frac{3}{2}E\left(\frac{q}{2}\right)+1\right)<E\left(\frac{q}{2}\right)^2 \\
2&\text{ if }\mathcal{M}od\left(d,\frac{3}{2}E\left(\frac{q}{2}\right)^2-\frac{3}{2}E\left(\frac{q}{2}\right)+1\right)>E\left(\frac{q}{2}\right)^2-1 \\
\end{array}\right.
\end{equation}
\\
\begin{equation}
\varphi(d,q)=E\left(\frac{d}{\frac{3}{2}E\left(\frac{q}{2}\right)^2-\frac{3}{2}E\left(\frac{q}{2}\right)+1}\right)-3E\left(\frac{1}{4}\left(1+E\left(\frac{d}{\frac{3}{2}E\left(\frac{q}{2}\right)^2-\frac{3}{2}E\left(\frac{q}{2}\right)+1}\right)\right)\right)
\end{equation}
\\
\begin{multline}
\gamma(d,q)=d+1-\left(\frac{3}{2}E\left(\frac{q}{2}\right)^2-\frac{3}{2}E\left(\frac{q}{2}\right)+1\right)E\left(\frac{d}{\frac{3}{2}E\left(\frac{q}{2}\right)^2-\frac{3}{2}E\left(\frac{q}{2}\right)+1}\right)\\-\frac{1}{2}\left(A(d,q)E\left(\frac{q}{2}\right)^2+A(d,q)(2-A(d,q))E\left(\frac{q}{2}\right)\right)
\end{multline}
\\
\normalsize
\begin{equation}
N(d,q)=E\left(\frac{1}{2}\left(1+\sqrt{8\gamma(d,q)-7}\right)\right)
\end{equation}
\\
\begin{equation}
M(d,q)=\gamma(d,q)-\frac{1}{2}\left(\left(N(d,q)-1\right)^2+N(d,q)-1\right)
\end{equation}
\\
\begin{multline}
S(p,q,r,\alpha,d)=\left(s_1\left(p,q,\alpha,A(d,q),N(d,q),M(d,q),B(d,q)\right)+1\right)\\ \left(s_2\left(q,A(d,q),N(d,q),B(d,q)\right)+1\right) 
\left(s_3\left(q,r,\alpha,A(d,q),N(d,q),M(d,q),B(d,q)\right)+1\right)
\end{multline}
\\
\begin{equation}
Y_0(q,a,n,b)=
\left\lbrace
\begin{array}{lcc}
b-1+q-2E\left(\frac{q}{2}\right)+2n&\text{ if }a=0 \\
b-1+q-2n&\text{ if }a=1 \\
b+q-2n-3&\text{ if }a=2 \\
\end{array}\right.
\end{equation}
\\
\begin{equation}
Y_1(p,q,\alpha,a,n,m,b)=
\left\lbrace
\begin{array}{lcc}
p+(2-\alpha)(1-b)+(2-\alpha)\left(2E\left(\frac{q}{2}\right)-2n\right)&\text{ if }a=0 \\
p+(\alpha-1)(1-b)+(-1)^\alpha(2m-1)+2n(2-\alpha)&\text{ if }a=1 \\
p+(2-\alpha)(1-b)+(-1)^\alpha 2m+(2n+2)(2-\alpha)&\text{ if }a=2 \\
\end{array}\right.
\end{equation}
\\
\begin{equation}
Y_2(q,r,\alpha,a,n,m,b)=
\left\lbrace
\begin{array}{lcc}
r+(\alpha-1)(1-b)+(\alpha-1)\left(2E\left(\frac{q}{2}\right)-2n\right)&\text{ if }a=0 \\
r+(2-\alpha)(1-b)+(-1)^{\alpha+1}(2m-1)+2n(\alpha-1)&\text{ if }a=1 \\
r+(\alpha-1)(1-b)+(-1)^{\alpha+1}2m+(2n+2)(\alpha-1)&\text{ if }a=2 \\
\end{array}\right.
\end{equation}
\\
\begin{equation}
H(q,a,n,m,b)=
\left\lbrace
\begin{array}{lcc}
\delta_{n,m}(-1)^{n+b+E\left(\frac{q}{2}\right)}&\text{ if }a=0 \\
(-1)^{n+b+\frac{1}{2}a(a-1)+\frac{1}{2}(a-1)(a-2)E\left(\frac{q}{2}\right)}\frac{(n-2+a)!}{(m-2+a)!(n-m)!}&\text{ if }a>0 \\
\end{array}\right.
\end{equation}
\\
\begin{equation}
L_0(q,r,\alpha,d,i)=E\left(\frac{i-1}{s_3(q,r,\alpha,A(d,q),N(d,q),M(d,q),B(d,q))+1}\right)
\end{equation}
\\
\footnotesize
\begin{equation}
L_1(q,r,\alpha,d,i)=E\left(\frac{i-1}{(s_3(q,r,\alpha,A(d,q),N(d,q),M(d,q),B(d,q))+1)(s_2(q,A(d,q),N(d,q),B(d,q))+1)}\right)
\end{equation}
\footnotesize
\begin{multline}
G(p,q,r,\alpha,d,i)=H(q,A(d,q),N(d,q),M(d,q),B(d,q))\\
\bigg\{\frac{1}{2}(\varphi(d,q)-1)(\varphi(d,q)-2)\left[L_0(q,r,\alpha,d,i)-
(s_2(q,A(d,q),N(d,q),B(d,q))+1)L_1(q,r,\alpha,d,i)\right] \\ + 
\varphi(d,q)(2-\varphi(d,q))L_1(q,r,\alpha,d,i)+
\frac{1}{2}\varphi(d,q)(\varphi(d,q)-1)\mathcal{M}od(i-1,s_3(q,r,\alpha,A(d,q),N(d,q),M(d,q),B(d,q))\\
+1)+1\bigg\}
\text{\scriptsize $\frac{Y_0(q,A(d,q),N(d,q),B(d,q))!Y_1(p,q,\alpha,A(d,q),N(d,q),M(d,q),B(d,q))!Y_2(q,r,\alpha,A(d,q),N(d,q),M(d,q),B(d,q))!}{p!q!r!}$}
\end{multline}
\normalsize
\begin{multline}
W_0(p,q,r,\alpha,d,i)=\frac{1}{2}(\varphi(d,q)-1)(\varphi(d,q)-2)+L_0(q,r,\alpha,d,i)-\\
(s_2(q,A(d,q),N(d,q),B(d,q))+1)L_1(q,r,\alpha,d,i)
\end{multline}
\begin{equation}
W_1(p,q,r,\alpha,d,i)=\varphi(d,q)(2-\varphi(d,q))+L_1(q,r,\alpha,d,i)
\end{equation}
\begin{multline}
W_2(p,q,r,\alpha,d,i)=\frac{1}{2}\varphi(d,q)(\varphi(d,q)-1)+\\ \mathcal{M}od(i-1,s_3(q,r,\alpha,A(d,q),N(d,q),M(d,q),B(d,q))+1)
\end{multline}
\begin{multline}
V_0(q,r,\alpha,d,i)=Y_0(q,A(d,q),N(d,q),B(d,q))-L_0(q,r,\alpha,d,i)+\\
(s_2(q,A(d,q),N(d,q),B(d,q))+1)L_1(q,r,\alpha,d,i)
\end{multline}
\begin{equation}
V_1(p,q,r,\alpha,d,i)=Y_1(p,q,\alpha,A(d,q),N(d,q),M(d,q),B(d,q))-L_1(q,r,\alpha,d,i)
\end{equation}
\begin{multline}
V_2(q,r,\alpha,d,i)=Y_2(q,r,\alpha,A(d,q),N(d,q),M(d,q),B(d,q))-\\
\mathcal{M}od(i-1,s_3(q,r,\alpha,A(d,q),N(d,q),M(d,q),B(d,q))+1)
\end{multline}
\normalsize
\begin{multline}
F_0(\omega,p,q,r,\alpha)=\sum_{n=1}^{E\left(\frac{q}{2}\right)}(-1)^{E\left(\frac{q}{2}\right)}{{E\left(\frac{q}{2}\right)-1}\choose{n-1}}\left(\frac{(X_\alpha+2n)!(X_{3-\alpha}+2E\left(\frac{q}{2}\right)-2n)!}{p!q!r!}\right. \\
u^\alpha_{\omega,p+(-1)^{\alpha+1}2n+(\alpha-1)2E\left(\frac{q}{2}\right),q-2E\left(\frac{q}{2}\right),r+(-1)^\alpha 2n+2(2-\alpha)E\left(\frac{q}{2}\right)}+
\frac{(X_\alpha+2n-1)!\left(X_{3-\alpha}+2E\left(\frac{q}{2}\right)-2n+1\right)!}{p!q!r!} \\
\left.u^{3-\alpha}_{\omega,p+(-1)^{\alpha+1}(2n-1)+(\alpha-1)2E\left(\frac{q}{2}\right),q-2E\left(\frac{q}{2}\right),r+(-1)^\alpha (2n-1)+2(2-\alpha)E\left(\frac{q}{2}\right)}\right)
\end{multline}
\footnotesize
\begin{multline}
\widetilde{F_0}(\omega,p,q,r,\alpha)=\frac{1}{\omega}\sum_{n=1}^{E\left(\frac{q}{2}\right)}(-1)^{n+1}\left(\frac{(q-2n+2)!(X_\alpha+2n-2)!}{X_\alpha !q!}g^\alpha_{\omega-1,p+(2-\alpha)(2n-2),q-2n+2,r+(\alpha-1)(2n-2)}-\right.\\ \left.\frac{(q-2n+1)!(X_\alpha+2n-1)!}{X_\alpha !q!}g^0_{\omega-1,p+(2-\alpha)(2n-1),q-2n+1,r+(\alpha-1)(2n-1)}\right)+\frac{1}{\omega}\sum_{n=1}^{E\left(\frac{q}{2}\right)-1}\sum_{m=1}^{n}(-1)^n\left({n-1\choose m-1}\right.\\ \left.
\frac{(X_{3-\alpha}+2m-1)!(q-2n)!(X_\alpha-2m+2n+1)!}{p!q!r!}
g^{3-\alpha}_{\omega-1,p+(-1)^\alpha(2m-1)+2n(2-\alpha),q-2n,r+(-1)^{\alpha+1}(2m-1)+2(\alpha-1)n}-{{n}\choose{m}}\right.\\
\left.
\frac{(X_{3-\alpha}+2m)!(q-1-2n)!(X_\alpha-2m+2n+1)!}{p!q!r!}g^0_{\omega-1,p+2m(-1)^\alpha+(2-\alpha)(2n+1),q-1-2n,r+2m(-1)^{\alpha+1}+(\alpha-1)(2n+1)}\right)\\
+\frac{1}{\omega}\sum_{n=1}^{E\left(\frac{q}{2}\right)-2}\sum_{m=1}^{n}(-1)^{n+1}\binom{n}{m}\frac{(X_{3-\alpha}+2m)!(q-2-2n)!(X_\alpha-2m+2n+2)!}{p!q!r!}\\g^\alpha_{\omega-1,p+2m(-1)^\alpha+(2-\alpha)(2n+2),q-2-2n,r+2m(-1)^{\alpha+1}+(\alpha-1)(2n+2)}
\end{multline}
\normalsize
Notations are required to simplify the following formulas. Three propositions are introduced:
$$P_1:``\omega>0,~p>1,~q>1,~r>1"$$
$$P_2:``\omega<0\text{ or }p<0\text{ or }q<0\text{ or }r<0"$$
$$P_3:``\text{$\omega=0$ or $0\leq p\leq 1$ or $0\leq q\leq 1$ or $0\leq r\leq 1$}"$$
For convenience, the notation $F=F(z,\omega,p,q,r,\alpha,n)$ is used for the next equation,
\footnotesize
\begin{equation}
F=
\left\lbrace
\begin{array}{lc}
F_0(\omega,p,q,r,\alpha)+\widetilde{F_0}(\omega,p,q,r,\alpha)&\text{ if $z=0$ and $P_1$ is true}\\
u^\alpha_{\omega,p,q,r}&\text{ if $z=0$ and $P_1$ is false}\\
\frac{\nu}{\omega}(X_{3-\alpha}+1)(X_{3-\alpha}+2)&\text{ if $z=1$ and $n=1$} \\
\frac{\nu}{\omega}(q+1)(q+2) &\text{ if $z=2$ and $n=1$} \\
-\frac{\nu}{\omega}(q+1)(X_\alpha+1)&\text{ if $z=3$ and $n=1$} \\
-\frac{\nu}{\omega}(p+1)(r+1)\mathcal{H}(q-4)&\text{ if $z=4$ and $n=1$} \\
\frac{\nu}{\omega}(-1)^{E\left(\frac{q}{2}\right)}\frac{\left(X_\alpha+2E\left(\frac{q}{2}\right)+1\right)!\left(q-2E\left(\frac{q}{2}\right)+1\right)}{X_\alpha !q!}&\text{ if $z=5$ and $n=1$} \\
\frac{\nu}{\omega}(-1)^{E\left(\frac{q}{2}\right)}\binom{E\left(\frac{q}{2}\right)}{n-1}\frac{(X_\alpha+2n-1)!\left(X_{3-\alpha}+2E\left(\frac{q}{2}\right)-2n+2\right)!\left(1+q-2E\left(\frac{q}{2}\right)\right)}{p!q!r!}&\text{ if $z=6$} \\
2\frac{\nu}{\omega}(-1)^{E\left(\frac{q}{2}\right)+1}\binom{E\left(\frac{q}{2}\right)-1}{n-1}\frac{(X_\alpha+2n)!\left(X_{3-\alpha}+2E\left(\frac{q}{2}\right)-2n\right)!\left(1+2q-4E\left(\frac{q}{2}\right)\right)}{p!q!r!}&\text{ if $z=7$} \\
2\mathcal{H}(q-4)\frac{\nu}{\omega}(-1)^{E\left(\frac{q}{2}\right)+1}\binom{E\left(\frac{q}{2}\right)-1}{n-1}\frac{(X_\alpha+2n-1)!\left(X_{3-\alpha}+2E\left(\frac{q}{2}\right)-2n+1\right)!\left(1+2q-4E\left(\frac{q}{2}\right)\right)}{p!q!r!}
&\text{ if $z=8$} \\
0 &\text{ otherwise}\\
\end{array}\right.
\end{equation}
\normalsize
\begin{equation}
D(p,q,r,\alpha ,\tau )=\sum_{j=1}^{9E\left(\frac{q}{2}\right)^2-9E\left(\frac{q}{2}\right)+6}\mathcal{H}\left(\frac{\tau-1}{\sum_{k=0}^{j-1}S(p,q,r,\alpha,k)}-1\right)
\end{equation}
\begin{equation}
\mathcal{I}(p,q,r,\alpha,\tau)=\tau-\sum_{j=1}^{9E\left(\frac{q}{2}\right)^2-9E\left(\frac{q}{2}\right)+6}S(p,q,r,\alpha,j-1)\mathcal{H}\left(\frac{\tau-1}{\sum_{k=0}^{j-1}S(p,q,r,\alpha,k)}-1\right)
\end{equation}
\small
\begin{multline}
\sigma(p,q,r,\alpha)=\left[2E\left(\frac{q}{2}\right)^4+(4X_\alpha+6X_{3-\alpha}+8)E\left(\frac{q}{2}\right)^3+(9X_\alpha X_{3-\alpha}+1.5X_\alpha+3X_{3-\alpha}-8)\right.\\
\left.E\left(\frac{q}{2}\right)^2-(9X_{3-\alpha}X_\alpha+5.5X_\alpha+9X_{3-\alpha}+2)E\left(\frac{q}{2}\right)+
6X_\alpha X_{3-\alpha}+6X_\alpha+9X_{3-\alpha}+9\right]q-\\
3.2E\left(\frac{q}{2}\right)^5-(6X_\alpha+8X_{3-\alpha}+8)E\left(\frac{q}{2}\right)^4+(6X_{3-\alpha}+6X_\alpha-10X_{3-\alpha}X_\alpha+24)E\left(\frac{q}{2}\right)^3+
\\
(14X_{3-\alpha}+9X_\alpha+22.5X_\alpha X_{3-\alpha}-7)E\left(\frac{q}{2}\right)^2-(9X_\alpha+12X_{3-\alpha}+12.5X_\alpha X_{3-\alpha}-2.2)E\left(\frac{q}{2}\right)+\\3(X_{3-\alpha}+1)(X_\alpha+1)-8E\left(\frac{q}{2}\right)
\label{sigma}
\end{multline}
\normalsize
\begin{equation}
\mathcal{T}(p,q,r,\alpha,j)=j+1-\sigma(p,q,r,\alpha)E\left(\frac{j}{\sigma(p,q,r,\alpha)}\right)
\end{equation}
\\
\begin{equation}
\chi(\omega,p,q,r,\alpha,i)=i-\sigma(p,q,r,\alpha)\omega
\end{equation}
\\
\begin{equation}
\mathcal{N}(\omega,p,q,r,\alpha,i)=E\left(\frac{\chi(\omega,p,q,r,\alpha,i)-1}{8}\right)+1
\end{equation}
\scriptsize
\begin{equation}
\mathcal{F}(p,q,r,\alpha,\omega)=
\left\lbrace
\begin{array}{lcc}
0&\text{ if $p<2$ or $q<2$ or $r<2$ or $\omega\leq 0$} \\
1&\text{ if $a=0$ and $q=2$}\\
\sigma(p,q,r,\alpha)\omega+8E\left(\frac{q}{2}\right) &\text{ if $a\neq 0$, $\omega>0$ and $p,q,r>1$}\\
\sigma(p+1,q-1,r,1)\omega+\sigma(p,q-1,r+1,2)\omega+16E\left(\frac{q-1}{2}\right)+1 &\text{ otherwise}\\
\end{array}\right.
\end{equation}
\normalsize
Two propositions need to be introduced afresh:
$$R_1:``i<\sigma(p,q,r,\alpha)\omega"$$
$$R_2:``i\leq\mathcal{F}(p+1,q-1,r,1,\omega)"$$
To minimize the size of the following formulas, "elif" will mean "else if".
\scriptsize
\begin{equation}
T_0(\omega,p,q,r,\alpha,i)=
\left\lbrace
\begin{array}{lcc}
-1&\text{ if $P_2$ is true}\\
q&\text{ if $P_3$ is true and $P_2$ is false}\\
W_0\left(p,q,r,\alpha,D(p,q,r,\alpha,\mathcal{T}(p,q,r,\alpha,i)),\mathcal{I}(p,q,r,\alpha,\mathcal{T}(p,q,r,\alpha,i))\right)&\text{ elif ($\alpha=1$ or $2$) and $R_1$ is true}\\
\Delta(\chi(\omega,p,q,r,\alpha,i)-8E_+\left(\frac{\chi(\omega,p,q,r,\alpha,i)-1}{8}\right),q) &\text{ elif ($\alpha=1$ or $2$) and $R_1$ is false}\\
T_0(\omega,p+1,q-1,r,1,i) &\text{ elif $\alpha=0$ and $R_2$ is true}\\
T_0(\omega,p,q-1,r+1,2,i-\mathcal{F}(p+1,q-1,r,1,\omega)-1) &\text{ otherwise}\\
\end{array}\right.
\end{equation}
\begin{equation}
T_1(\omega,p,q,r,\alpha,i)=
\left\lbrace
\begin{array}{lcc}
-1&\text{ if $P_2$ is true}\\
p&\text{ if $P_3$ is true and $P_2$ is false}\\
W_1\left(p,q,r,\alpha,D(p,q,r,\alpha,\mathcal{T}(p,q,r,\alpha,i)),\mathcal{I}(p,q,r,\alpha,\mathcal{T}(p,q,r,\alpha,i))\right)&\text{ elif ($\alpha=1$ or $2$) and $R_1$ is true}\\
\xi(\chi(\omega,p,q,r,\alpha,i)-8E_+\left(\frac{\chi(\omega,p,q,r,\alpha,i)-1}{8}\right),p,q,\alpha,\mathcal{N}(\omega,p,q,r,\alpha,i)) &\text{ elif ($\alpha=1$ or $2$) and $R_1$ is false}\\
T_1(\omega,p+1,q-1,r,1,i) &\text{ elif $\alpha=0$ and $R_2$ is true}\\
T_1(\omega,p,q-1,r+1,2,i-\mathcal{F}(p+1,q-1,r,1,\omega)-1) &\text{ otherwise}\\
\end{array}\right.
\end{equation}
\begin{equation}
T_2(\omega,p,q,r,\alpha,i)=
\left\lbrace
\begin{array}{lcc}
-1&\text{ if $P_2$ is true}\\
r&\text{ if $P_3$ is true and $P_2$ is false}\\
W_2\left(p,q,r,\alpha,D(p,q,r,\alpha,\mathcal{T}(p,q,r,\alpha,i)),\mathcal{I}(p,q,r,\alpha,\mathcal{T}(p,q,r,\alpha,i))\right)&\text{ elif ($\alpha=1$ or $2$) and $R_1$ is true}\\
\epsilon(\chi(\omega,p,q,r,\alpha,i)-8E_+\left(\frac{\chi(\omega,p,q,r,\alpha,i)-1}{8}\right),q,r,\alpha,\mathcal{N}(\omega,p,q,r,\alpha,i)) &\text{ elif ($\alpha=1$ or $2$) and $R_1$ is false}\\
T_2(\omega,p+1,q-1,r,1,i) &\text{ elif $\alpha=0$ and $R_2$ is true}\\
T_2(\omega,p,q-1,r+1,2,i-\mathcal{F}(p+1,q-1,r,1,\omega)-1) &\text{ otherwise}\\
\end{array}\right.
\end{equation}
\begin{equation}
\beta(\omega,p,q,r,\alpha,i)=
\left\lbrace
\begin{array}{lcc}
\alpha&\text{ if $P_2\vee P_3$ is true}\\
B(d,q)(\alpha+A(d,q)(A(d,q)-2)(2\alpha-3))&\text{ elif ($\alpha=1$ or $2$) and $R_1$ is true}\\
\kappa(\chi(\omega,p,q,r,\alpha,i)-8E_+\left(\frac{\chi(\omega,p,q,r,\alpha,i)-1}{8}\right),\alpha) &\text{ elif ($\alpha=1$ or $2$) and $R_1$ is false}\\
\beta(\omega,p+1,q-1,r,1,i) &\text{ elif $\alpha=0$ and $R_2$ is true}\\
\beta(\omega,p,q-1,r+1,2,i-\mathcal{F}(p+1,q-1,r,1,\omega)-1) &\text{ otherwise}\\
\end{array}\right.
\end{equation}
\begin{equation}
\mathcal{E}(\omega,p,q,r,\alpha,i)=
\left\lbrace
\begin{array}{lcc}
-1&\text{ if $P_2$ is true}\\
\omega&\text{ if $P_3$ is true and $P_2$ is false}\\
E\left(\frac{i}{\sigma(p,q,r,\alpha)}\right)&\text{ elif ($\alpha=1$ or $2$) and $E\left(\frac{i}{\sigma(p,q,r,\alpha)}\right)<\omega$}\\
\omega-1 &\text{ elif ($\alpha=1$ or $2$) and $\omega\leq E\left(\frac{i}{\sigma(p,q,r,\alpha)}\right)$}\\
\mathcal{E}(\omega,p+1,q-1,r,1,i) &\text{ elif $\alpha=0$ and $R_2$ is true}\\
\mathcal{E}(\omega,p,q-1,r+1,2,i-\mathcal{F}(p+1,q-1,r,1,\omega)-1) &\text{ otherwise}\\
\end{array}\right.
\end{equation}
\begin{equation}
Z_0(\omega,p,q,r,\alpha,i)=
\left\lbrace
\begin{array}{lcc}
-1&\text{ if $P_2\vee P_3$ is true}\\
V_0\left(q,r,\alpha,D(p,q,r,\alpha,\mathcal{T}(p,q,r,\alpha,i)),\mathcal{I}(p,q,r,\alpha,\mathcal{T}(p,q,r,\alpha,i))\right)&\text{ elif $\alpha=1$ or $2$}\\
Z_0(\omega,p+1,q-1,r,1,i) &\text{ elif $\alpha=0$ and $R_2$ is true}\\
Z_0(\omega,p,q-1,r+1,2,i-\mathcal{F}(p+1,q-1,r,1,\omega)-1) &\text{ otherwise}\\
\end{array}\right.
\end{equation}
\begin{equation}
Z_1(\omega,p,q,r,\alpha,i)=
\left\lbrace
\begin{array}{lcc}
-1&\text{ if $P_2\vee P_3$ is true}\\
V_1\left(p,q,r,\alpha,D(p,q,r,\alpha,\mathcal{T}(p,q,r,\alpha,i)),\mathcal{I}(p,q,r,\alpha,\mathcal{T}(p,q,r,\alpha,i))\right)&\text{ elif $\alpha=1$ or $2$}\\
Z_1(\omega,p+1,q-1,r,1,i) &\text{ elif $\alpha=0$ and $R_2$ is true}\\
Z_1(\omega,p,q-1,r+1,2,i-\mathcal{F}(p+1,q-1,r,1,\omega)-1) &\text{ otherwise}\\
\end{array}\right.
\end{equation}
\begin{equation}
Z_2(\omega,p,q,r,\alpha,i)=
\left\lbrace
\begin{array}{lcc}
-1&\text{ if $P_2\vee P_3$ is true}\\
V_2\left(p,q,r,\alpha,D(p,q,r,\alpha,\mathcal{T}(p,q,r,\alpha,i)),\mathcal{I}(p,q,r,\alpha,\mathcal{T}(p,q,r,\alpha,i))\right)&\text{ elif $\alpha=1$ or $2$}\\
Z_2(\omega,p+1,q-1,r,1,i) &\text{ elif $\alpha=0$ and $R_2$ is true}\\
Z_2(\omega,p,q-1,r+1,2,i-\mathcal{F}(p+1,q-1,r,1,\omega)-1) &\text{ otherwise}\\
\end{array}\right.
\end{equation}
\scriptsize
\begin{equation}
\Phi(\omega,p,q,r,\alpha,i)=
\left\lbrace
\begin{array}{lcc}
\alpha&\text{ if $P_2\vee P_3$ is true}\\
\varphi\left(D(p,q,r,\alpha,\mathcal{T}(p,q,r,\alpha,i)),q\right)&\text{ elif $\alpha=1$ or $2$}\\
\Phi(\omega,p+1,q-1,r,1,i) &\text{ elif $\alpha=0$ and $R_2$ is true}\\
\Phi(\omega,p,q-1,r+1,2,i-\mathcal{F}(p+1,q-1,r,1,\omega)-1) &\text{ otherwise}\\
\end{array}\right.
\end{equation}
\begin{equation}
\Upsilon(\omega,p,q,r,\alpha,i)=
\left\lbrace
\begin{array}{lcc}
-1&\text{ if $P_2\vee P_3$ is true}\\
\omega-1-E\left(\frac{i}{\sigma(p,q,r,\alpha)}\right)&\text{ elif $\alpha=1$ or $2$}\\
\Upsilon(\omega,p+1,q-1,r,1,i) &\text{ elif $\alpha=0$ and $R_2$ is true}\\
\Upsilon(\omega,p,q-1,r+1,2,i-\mathcal{F}(p+1,q-1,r,1,\omega)-1) &\text{ otherwise}\\
\end{array}\right.
\end{equation}
\begin{equation}
\Gamma(\omega,p,q,r,\alpha,i)=
\left\lbrace
\begin{array}{lcc}
1&\text{ if $P_3$ is true}\\
\frac{1}{\omega}G\left(p,q,r,\alpha,D(p,q,r,\alpha,\mathcal{T}(p,q,r,\alpha,i)),\mathcal{I}(p,q,r,\alpha,\mathcal{T}(p,q,r,\alpha,i))\right)&\text{ elif ($\alpha=1$ or $2$) and $R_1$ is true}\\
F(\chi(\omega,p,q,r,\alpha,i)-8E_+\left(\frac{\chi(\omega,p,q,r,\alpha,i)-1}{8}\right),\omega,p,q,r,\alpha,\mathcal{N}(\omega,p,q,r,\alpha,i)) &\text{ elif ($\alpha=1$ or $2$) and $R_1$ is false}\\
-\frac{p+1}{q}\Gamma(\omega,p+1,q-1,r,1,i) &\text{ elif $\alpha=0$ and $R_2$ is true}\\
-\frac{r+1}{q}\Gamma(\omega,p,q-1,r+1,2,i-\mathcal{F}(p+1,q-1,r,1,\omega)-1) &\text{ elif $\alpha=0$ and $R_2$ is false}\\
\end{array}\right.
\label{Gamma}
\end{equation}
\begin{equation}
\eta(p,q,r,\alpha,\omega,ind)=
\left\lbrace
\begin{array}{lcc}
0&\text{ if $P_2\vee P_3$ is true}\\
1&\text{ if $\alpha=0$ and $q=2$} \\
\mathcal{F}(p,q,r,\alpha,\omega) &\text{ if $P_2\vee P_3$ is false and $ind=0$} \\
\sum\limits_{k=0}^{\mathcal{F}(p,q,r,\alpha,\omega)}\left\{\left[\eta(T_1(\omega,p,q,r,\alpha,k),T_0(\omega,p,q,r,\alpha,k),T_2(\omega,p,q,r,\alpha,k),\right.\right.\\ \left.
\beta(\omega,p,q,r,\alpha,k),
\mathcal{E}(\omega,p,q,r,\alpha,k),ind-1)+1\right]&\text{ otherwise}\\
\left[\eta(Z_1(\omega,p,q,r,\alpha,k),Z_0(\omega,p,q,r,\alpha,k),Z_2(\omega,p,q,r,\alpha,k),\right. \\ \left.\left.\Phi(\omega,p,q,r,\alpha,k),\Upsilon(\omega,p,q,r,\alpha,k),ind-1)+1\right]
\right\}-1 
\end{array}\right. 
\label{eta}
\end{equation}
\footnotesize
\begin{equation}
\widetilde{\eta_1}(p,q,r,\alpha,\omega,k,ind)=
\left\lbrace
\begin{array}{lcc}
\mathcal{F}(T_1(\omega,p,q,r,\alpha,k),T_0(\omega,p,q,r,\alpha,k),T_2(\omega,p,q,r,\alpha,k),\\
\beta(\omega,p,q,r,\alpha,k),
\mathcal{E}(\omega,p,q,r,\alpha,k)) &\text{ if $ind=0$}\\
\eta(T_1(\omega,p,q,r,\alpha,k),T_0(\omega,p,q,r,\alpha,k),T_2(\omega,p,q,r,\alpha,k),\\
\beta(\omega,p,q,r,\alpha,k),
\mathcal{E}(\omega,p,q,r,\alpha,k),ind) &\text{ if $ind>0$}\\
\end{array}\right.
\end{equation}
\begin{equation}
\widetilde{\eta_2}(p,q,r,\alpha,\omega,k,ind)=
\left\lbrace
\begin{array}{lcc}
\mathcal{F}(Z_1(\omega,p,q,r,\alpha,k),Z_0(\omega,p,q,r,\alpha,k),Z_2(\omega,p,q,r,\alpha,k),\\
\Phi(\omega,p,q,r,\alpha,k),
\Upsilon(\omega,p,q,r,\alpha,k)) &\text{ if $ind=0$}\\
\eta(Z_1(\omega,p,q,r,\alpha,k),Z_0(\omega,p,q,r,\alpha,k),Z_2(\omega,p,q,r,\alpha,k),\\
\Phi(\omega,p,q,r,\alpha,k),
\Upsilon(\omega,p,q,r,\alpha,k),ind) &\text{ if $ind>0$}\\
\end{array}\right.
\end{equation}
\normalsize
\begin{equation}
\mathcal{S}\widetilde{\eta}(j,p,q,r,\alpha,\omega,ind)=\sum\limits_{k=0}^{j-1}[\widetilde{\eta_1}(p,q,r,\alpha,\omega,k,ind)+1][\widetilde{\eta_2}(p,q,r,\alpha,\omega,k,ind)+1]
\end{equation}
\scriptsize
\begin{equation}
\theta(i,p,q,r,\alpha,\omega,n,ind)=
\left\lbrace
\begin{array}{lcc}
0 &\text{ if $P_2\vee P_3$ is true}\\
i &\text{ elif $ind=0$}\\
j \text{ such as } i\in\left[\mathcal{S}\widetilde{\eta}(j,p,q,r,\alpha,\omega,ind-1);\mathcal{S}\widetilde{\eta}(j+1,p,q,r,\alpha,\omega,ind-1)\right] &\text{ elif $ind>0$ and $n=0$}\\
E\left(\frac{i-\mathcal{S}\widetilde{\eta}(\theta(i,p,q,r,\alpha,\omega,0,ind),p,q,r,\alpha,\omega,ind-1)}{
\widetilde{\eta_2}(p,q,r,\alpha,\omega,\theta(i,p,q,r,\alpha,\omega,0,ind),ind-1)+1}\right) &\text{ elif $ind>0$ and $n=1$} \\
\mathcal{R}(i-\mathcal{S}\widetilde{\eta}(\theta(i,p,q,r,\alpha,\omega,0,ind),p,q,r,\alpha,\omega,ind-1),\\
\widetilde{\eta_2}(p,q,r,\alpha,\omega,\theta(i,p,q,r,\alpha,\omega,0,ind),ind-1)+1) &\text{ elif $ind>0$ and $n=2$} \\
\end{array}\right.
\label{theta}
\end{equation}
\begin{equation}
C_0(m,h,l,i,p,q,r,\alpha,\omega)=
\begin{cases}
q~~~~~~~~\text{  if $E\left(\frac{ln(m+1)}{ln(2)}\right)-h<1$ or $l>E\left(\frac{ln(m+1)}{ln(2)}\right)-1-h$}\\
\\
T_0(C_\omega(m,h,l+1,i,p,q,r,\alpha,\omega),C_1(m,h,l+1,i,p,q,r,\alpha,\omega),&\text{ } \\ 
C_0(m,h,l+1,i,p,q,r,\alpha,\omega),C_2(m,h,l+1,i,p,q,r,\alpha,\omega),&\text{ elif $Ip\left(E\left(\frac{m+1-2^{l+h}}{2^{l+h}}\right)\right)$}\\ 
C_\alpha(m,h,l+1,i,p,q,r,\alpha,\omega),\Theta\left(E\left(\frac{m+1-2^{l+1+h}}{2^{l+1+h}}\right),i,p,q,r,\alpha,\omega,0\right))&\text{$~~~~~~~~~~~~~~=1$} \\
\\
Z_0(C_\omega(m,h,l+1,i,p,q,r,\alpha,\omega),C_1(m,h,l+1,i,p,q,r,\alpha,\omega),&\text{ } \\ 
C_0(m,h,l+1,i,p,q,r,\alpha,\omega),C_2(m,h,l+1,i,p,q,r,\alpha,\omega),&\text{ elif $Ip\left(E\left(\frac{m+1-2^{l+h}}{2^{l+h}}\right)\right)$}\\ 
C_\alpha(m,h,l+1,i,p,q,r,\alpha,\omega),\Theta\left(E\left(\frac{m+1-2^{l+1+h}}{2^{l+1+h}}\right),i,p,q,r,\alpha,\omega,0\right))&\text{$~~~~~~~~~~~~~~\neq 1$} \\
\end{cases}
\label{DebC}
\end{equation}
\begin{equation}
C_1(m,h,l,i,p,q,r,\alpha,\omega)=
\begin{cases}
p~~~~~~~~\text{  if $E\left(\frac{ln(m+1)}{ln(2)}\right)-h<1$ or $l>E\left(\frac{ln(m+1)}{ln(2)}\right)-1-h$}\\
\\
T_1(C_\omega(m,h,l+1,i,p,q,r,\alpha,\omega),C_1(m,h,l+1,i,p,q,r,\alpha,\omega),&\text{ } \\ 
C_0(m,h,l+1,i,p,q,r,\alpha,\omega),C_2(m,h,l+1,i,p,q,r,\alpha,\omega),&\text{ elif $Ip\left(E\left(\frac{m+1-2^{l+h}}{2^{l+h}}\right)\right)$}\\ 
C_\alpha(m,h,l+1,i,p,q,r,\alpha,\omega),\Theta\left(E\left(\frac{m+1-2^{l+1+h}}{2^{l+1+h}}\right),i,p,q,r,\alpha,\omega,0\right))&\text{$~~~~~~~~~~~~~~=1$} \\
\\
Z_1(C_\omega(m,h,l+1,i,p,q,r,\alpha,\omega),C_1(m,h,l+1,i,p,q,r,\alpha,\omega),&\text{ } \\ 
C_0(m,h,l+1,i,p,q,r,\alpha,\omega),C_2(m,h,l+1,i,p,q,r,\alpha,\omega),&\text{ elif $Ip\left(E\left(\frac{m+1-2^{l+h}}{2^{l+h}}\right)\right)$}\\ 
C_\alpha(m,h,l+1,i,p,q,r,\alpha,\omega),\Theta\left(E\left(\frac{m+1-2^{l+1+h}}{2^{l+1+h}}\right),i,p,q,r,\alpha,\omega,0\right))&\text{$~~~~~~~~~~~~~~\neq 1$} \\
\end{cases}
\end{equation}
\begin{equation}
C_2(m,h,l,i,p,q,r,\alpha,\omega)=
\begin{cases}
r~~~~~~~~\text{  if $E\left(\frac{ln(m+1)}{ln(2)}\right)-h<1$ or $l>E\left(\frac{ln(m+1)}{ln(2)}\right)-1-h$}\\
\\
T_2(C_\omega(m,h,l+1,i,p,q,r,\alpha,\omega),C_1(m,h,l+1,i,p,q,r,\alpha,\omega),&\text{ } \\ 
C_0(m,h,l+1,i,p,q,r,\alpha,\omega),C_2(m,h,l+1,i,p,q,r,\alpha,\omega),&\text{ elif $Ip\left(E\left(\frac{m+1-2^{l+h}}{2^{l+h}}\right)\right)$}\\ 
C_\alpha(m,h,l+1,i,p,q,r,\alpha,\omega),\Theta\left(E\left(\frac{m+1-2^{l+1+h}}{2^{l+1+h}}\right),i,p,q,r,\alpha,\omega,0\right))&\text{$~~~~~~~~~~~~~~=1$} \\
\\
Z_2(C_\omega(m,h,l+1,i,p,q,r,\alpha,\omega),C_1(m,h,l+1,i,p,q,r,\alpha,\omega),&\text{ } \\ 
C_0(m,h,l+1,i,p,q,r,\alpha,\omega),C_2(m,h,l+1,i,p,q,r,\alpha,\omega),&\text{ elif $Ip\left(E\left(\frac{m+1-2^{l+h}}{2^{l+h}}\right)\right)$}\\ 
C_\alpha(m,h,l+1,i,p,q,r,\alpha,\omega),\Theta\left(E\left(\frac{m+1-2^{l+1+h}}{2^{l+1+h}}\right),i,p,q,r,\alpha,\omega,0\right))&\text{$~~~~~~~~~~~~~~\neq 1$} \\
\end{cases}
\end{equation}
\begin{equation}
C_\alpha(m,h,l,i,p,q,r,\alpha,\omega)=
\begin{cases}
\alpha~~~~~~~~\text{  if $E\left(\frac{ln(m+1)}{ln(2)}\right)-h<1$ or $l>E\left(\frac{ln(m+1)}{ln(2)}\right)-1-h$}\\
\\
\beta(C_\omega(m,h,l+1,i,p,q,r,\alpha,\omega),C_1(m,h,l+1,i,p,q,r,\alpha,\omega),&\text{ } \\ 
C_0(m,h,l+1,i,p,q,r,\alpha,\omega),C_2(m,h,l+1,i,p,q,r,\alpha,\omega),&\text{ elif $Ip\left(E\left(\frac{m+1-2^{l+h}}{2^{l+h}}\right)\right)$}\\ 
C_\alpha(m,h,l+1,i,p,q,r,\alpha,\omega),\Theta\left(E\left(\frac{m+1-2^{l+1+h}}{2^{l+1+h}}\right),i,p,q,r,\alpha,\omega,0\right))&\text{$~~~~~~~~~~~~~~=1$} \\
\\
\Phi(C_\omega(m,h,l+1,i,p,q,r,\alpha,\omega),C_1(m,h,l+1,i,p,q,r,\alpha,\omega),&\text{ } \\ 
C_0(m,h,l+1,i,p,q,r,\alpha,\omega),C_2(m,h,l+1,i,p,q,r,\alpha,\omega),&\text{ elif $Ip\left(E\left(\frac{m+1-2^{l+h}}{2^{l+h}}\right)\right)$}\\ 
C_\alpha(m,h,l+1,i,p,q,r,\alpha,\omega),\Theta\left(E\left(\frac{m+1-2^{l+1+h}}{2^{l+1+h}}\right),i,p,q,r,\alpha,\omega,0\right))&\text{$~~~~~~~~~~~~~~\neq 1$} \\
\end{cases}
\end{equation}
\begin{equation}
C_\omega(m,h,l,i,p,q,r,\alpha,\omega)=
\begin{cases}
\omega~~~~~~~~\text{  if $E\left(\frac{ln(m+1)}{ln(2)}\right)-h<1$ or $l>E\left(\frac{ln(m+1)}{ln(2)}\right)-1-h$}\\
\\
\mathcal{E}(C_\omega(m,h,l+1,i,p,q,r,\alpha,\omega),C_1(m,h,l+1,i,p,q,r,\alpha,\omega),&\text{ } \\ 
C_0(m,h,l+1,i,p,q,r,\alpha,\omega),C_2(m,h,l+1,i,p,q,r,\alpha,\omega),&\text{ elif $Ip\left(E\left(\frac{m+1-2^{l+h}}{2^{l+h}}\right)\right)$}\\ 
C_\alpha(m,h,l+1,i,p,q,r,\alpha,\omega),\Theta\left(E\left(\frac{m+1-2^{l+1+h}}{2^{l+1+h}}\right),i,p,q,r,\alpha,\omega,0\right))&\text{$~~~~~~~~~~~~~~=1$} \\
\\
\Upsilon(C_\omega(m,h,l+1,i,p,q,r,\alpha,\omega),C_1(m,h,l+1,i,p,q,r,\alpha,\omega),&\text{ } \\ 
C_0(m,h,l+1,i,p,q,r,\alpha,\omega),C_2(m,h,l+1,i,p,q,r,\alpha,\omega),&\text{ elif $Ip\left(E\left(\frac{m+1-2^{l+h}}{2^{l+h}}\right)\right)$}\\ 
C_\alpha(m,h,l+1,i,p,q,r,\alpha,\omega),\Theta\left(E\left(\frac{m+1-2^{l+1+h}}{2^{l+1+h}}\right),i,p,q,r,\alpha,\omega,0\right))&\text{$~~~~~~~~~~~~~~\neq 1$} \\
\end{cases}
\label{FinC}
\end{equation}
\begin{equation}
\Theta(m,i,p,q,r,\alpha,\omega,h)=
\begin{cases}
\theta(i,p,q,r,\alpha,\omega,Ip_b(h,m),\omega-1)&\text{ if $h=E\left(\frac{ln(m+1)}{ln(2)}\right)$}\\
\theta(\Theta(m,i,p,q,r,\alpha,\omega,h+1),C_1(m,h,0,i,p,q,r,\alpha,\omega),\\ C_0(m,h,0,i,p,q,r,\alpha,\omega),
C_2(m,h,0,i,p,q,r,\alpha,\omega),C_\alpha(m,h,0,i,p,q,r,\alpha,\omega),&\text{ otherwise}\\ C_\omega(m,h,0,i,p,q,r,\alpha,\omega),
Ip_b(h,m),\omega-1+h-E\left(\frac{ln(m+1)}{ln(2)}\right))
\end{cases}
\end{equation}
\begin{equation}
\rho_0(i,l,k,j,p,q,r,\alpha,\omega)=
\begin{cases}
q&\text{ if $j<2$}\\
\\
T_0(\rho_\omega(i,l+1,k,j,p,q,r,\alpha,\omega),\rho_1(i,l+1,k,j,p,q,r,\alpha,\omega),&\text{ } \\ 
\rho_0(i,l+1,k,j,p,q,r,\alpha,\omega),\rho_2(i,l+1,k,j,p,q,r,\alpha,\omega),&\text{ elif $Ip\left(E\left(\frac{k-1+2^{j-1}}{2^{l-1}}-1\right)\right)$}\\ 
\rho_\alpha(i,l+1,k,j,p,q,r,\alpha,\omega),\Theta\left(E\left(\frac{k-1+2^{j-1}-2^l}{2^l}\right),i,p,q,r,\alpha,\omega,0\right))&\text{$~~~~~~~~~~=1$ and $l<j-1$} \\
\\
Z_0(\rho_\omega(i,l+1,k,j,p,q,r,\alpha,\omega),\rho_1(i,l+1,k,j,p,q,r,\alpha,\omega),&\text{ } \\ 
\rho_0(i,l+1,k,j,p,q,r,\alpha,\omega),\rho_2(i,l+1,k,j,p,q,r,\alpha,\omega),&\text{ elif $Ip\left(E\left(\frac{k-1+2^{j-1}}{2^{l-1}}-1\right)\right)$}\\ 
\rho_\alpha(i,l+1,k,j,p,q,r,\alpha,\omega),\Theta\left(E\left(\frac{k-1+2^{j-1}-2^l}{2^l}\right),i,p,q,r,\alpha,\omega,0\right))&\text{$~~~~~~~\neq 1$ and $l<j-1$} \\
\\
T_0(\omega,p,q,r,\alpha,\Theta\left(E\left(\frac{k-1+2^{j-1}-2^l}{2^l}\right),i,p,q,r,\alpha,\omega,0\right))&\text{ elif $Ip\left(E\left(\frac{k-1+2^{j-1}}{2^{l-1}}-1\right)\right)$}\\
&\text{$~~~~~~~=1$ and $j-1\leq l$}\\
\\
Z_0(\omega,p,q,r,\alpha,\Theta\left(E\left(\frac{k-1+2^{j-1}-2^l}{2^l}\right),i,p,q,r,\alpha,\omega,0\right))&\text{ elif $Ip\left(E\left(\frac{k-1+2^{j-1}}{2^{l-1}}-1\right)\right)$}\\
&\text{$~~~~~~~=1$ and $j-1\leq l$}\\
\end{cases}
\end{equation}
\begin{equation}
\rho_1(i,l,k,j,p,q,r,\alpha,\omega)=
\begin{cases}
p&\text{ if $j<2$}\\
\\
T_1(\rho_\omega(i,l+1,k,j,p,q,r,\alpha,\omega),\rho_1(i,l+1,k,j,p,q,r,\alpha,\omega),&\text{ } \\ 
\rho_0(i,l+1,k,j,p,q,r,\alpha,\omega),\rho_2(i,l+1,k,j,p,q,r,\alpha,\omega),&\text{ elif $Ip\left(E\left(\frac{k-1+2^{j-1}}{2^{l-1}}-1\right)\right)$}\\ 
\rho_\alpha(i,l+1,k,j,p,q,r,\alpha,\omega),\Theta\left(E\left(\frac{k-1+2^{j-1}-2^l}{2^l}\right),i,p,q,r,\alpha,\omega,0\right))&\text{$~~~~~~~~~~=1$ and $l<j-1$} \\
\\
Z_1(\rho_\omega(i,l+1,k,j,p,q,r,\alpha,\omega),\rho_1(i,l+1,k,j,p,q,r,\alpha,\omega),&\text{ } \\ 
\rho_0(i,l+1,k,j,p,q,r,\alpha,\omega),\rho_2(i,l+1,k,j,p,q,r,\alpha,\omega),&\text{ elif $Ip\left(E\left(\frac{k-1+2^{j-1}}{2^{l-1}}-1\right)\right)$}\\ 
\rho_\alpha(i,l+1,k,j,p,q,r,\alpha,\omega),\Theta\left(E\left(\frac{k-1+2^{j-1}-2^l}{2^l}\right),i,p,q,r,\alpha,\omega,0\right))&\text{$~~~~~~~\neq 1$ and $l<j-1$} \\
\\
T_1(\omega,p,q,r,\alpha,\Theta\left(E\left(\frac{k-1+2^{j-1}-2^l}{2^l}\right),i,p,q,r,\alpha,\omega,0\right))&\text{ elif $Ip\left(E\left(\frac{k-1+2^{j-1}}{2^{l-1}}-1\right)\right)$}\\
&\text{$~~~~~~~=1$ and $j-1\leq l$}\\
\\
Z_1(\omega,p,q,r,\alpha,\Theta\left(E\left(\frac{k-1+2^{j-1}-2^l}{2^l}\right),i,p,q,r,\alpha,\omega,0\right))&\text{ elif $Ip\left(E\left(\frac{k-1+2^{j-1}}{2^{l-1}}-1\right)\right)$}\\
&\text{$~~~~~~~=1$ and $j-1\leq l$}\\
\end{cases}
\end{equation}
\begin{equation}
\rho_2(i,l,k,j,p,q,r,\alpha,\omega)=
\begin{cases}
r&\text{ if $j<2$}\\
\\
T_2(\rho_\omega(i,l+1,k,j,p,q,r,\alpha,\omega),\rho_1(i,l+1,k,j,p,q,r,\alpha,\omega),&\text{ } \\ 
\rho_0(i,l+1,k,j,p,q,r,\alpha,\omega),\rho_2(i,l+1,k,j,p,q,r,\alpha,\omega),&\text{ elif $Ip\left(E\left(\frac{k-1+2^{j-1}}{2^{l-1}}-1\right)\right)$}\\ 
\rho_\alpha(i,l+1,k,j,p,q,r,\alpha,\omega),\Theta\left(E\left(\frac{k-1+2^{j-1}-2^l}{2^l}\right),i,p,q,r,\alpha,\omega,0\right))&\text{$~~~~~~~~~~=1$ and $l<j-1$} \\
\\
Z_2(\rho_\omega(i,l+1,k,j,p,q,r,\alpha,\omega),\rho_1(i,l+1,k,j,p,q,r,\alpha,\omega),&\text{ } \\ 
\rho_0(i,l+1,k,j,p,q,r,\alpha,\omega),\rho_2(i,l+1,k,j,p,q,r,\alpha,\omega),&\text{ elif $Ip\left(E\left(\frac{k-1+2^{j-1}}{2^{l-1}}-1\right)\right)$}\\ 
\rho_\alpha(i,l+1,k,j,p,q,r,\alpha,\omega),\Theta\left(E\left(\frac{k-1+2^{j-1}-2^l}{2^l}\right),i,p,q,r,\alpha,\omega,0\right))&\text{$~~~~~~~\neq 1$ and $l<j-1$} \\
\\
T_2(\omega,p,q,r,\alpha,\Theta\left(E\left(\frac{k-1+2^{j-1}-2^l}{2^l}\right),i,p,q,r,\alpha,\omega,0\right))&\text{ elif $Ip\left(E\left(\frac{k-1+2^{j-1}}{2^{l-1}}-1\right)\right)$}\\
&\text{$~~~~~~~=1$ and $j-1\leq l$}\\
\\
Z_2(\omega,p,q,r,\alpha,\Theta\left(E\left(\frac{k-1+2^{j-1}-2^l}{2^l}\right),i,p,q,r,\alpha,\omega,0\right))&\text{ elif $Ip\left(E\left(\frac{k-1+2^{j-1}}{2^{l-1}}-1\right)\right)$}\\
&\text{$~~~~~~~=1$ and $j-1\leq l$}\\
\end{cases}
\end{equation}
\begin{equation}
\rho_\alpha(i,l,k,j,p,q,r,\alpha,\omega)=
\begin{cases}
\alpha&\text{ if $j<2$}\\
\\
\beta(\rho_\omega(i,l+1,k,j,p,q,r,\alpha,\omega),\rho_1(i,l+1,k,j,p,q,r,\alpha,\omega),&\text{ } \\ 
\rho_0(i,l+1,k,j,p,q,r,\alpha,\omega),\rho_2(i,l+1,k,j,p,q,r,\alpha,\omega),&\text{ elif $Ip\left(E\left(\frac{k-1+2^{j-1}}{2^{l-1}}-1\right)\right)$}\\ 
\rho_\alpha(i,l+1,k,j,p,q,r,\alpha,\omega),\Theta\left(E\left(\frac{k-1+2^{j-1}-2^l}{2^l}\right),i,p,q,r,\alpha,\omega,0\right))&\text{$~~~~~~~~~~=1$ and $l<j-1$} \\
\\
\Phi(\rho_\omega(i,l+1,k,j,p,q,r,\alpha,\omega),\rho_1(i,l+1,k,j,p,q,r,\alpha,\omega),&\text{ } \\ 
\rho_0(i,l+1,k,j,p,q,r,\alpha,\omega),\rho_2(i,l+1,k,j,p,q,r,\alpha,\omega),&\text{ elif $Ip\left(E\left(\frac{k-1+2^{j-1}}{2^{l-1}}-1\right)\right)$}\\ 
\rho_\alpha(i,l+1,k,j,p,q,r,\alpha,\omega),\Theta\left(E\left(\frac{k-1+2^{j-1}-2^l}{2^l}\right),i,p,q,r,\alpha,\omega,0\right))&\text{$~~~~~~~\neq 1$ and $l<j-1$} \\
\\
\beta(\omega,p,q,r,\alpha,\Theta\left(E\left(\frac{k-1+2^{j-1}-2^l}{2^l}\right),i,p,q,r,\alpha,\omega,0\right))&\text{ elif $Ip\left(E\left(\frac{k-1+2^{j-1}}{2^{l-1}}-1\right)\right)$}\\
&\text{$~~~~~~~=1$ and $j-1\leq l$}\\
\\
\Phi(\omega,p,q,r,\alpha,\Theta\left(E\left(\frac{k-1+2^{j-1}-2^l}{2^l}\right),i,p,q,r,\alpha,\omega,0\right))&\text{ elif $Ip\left(E\left(\frac{k-1+2^{j-1}}{2^{l-1}}-1\right)\right)$}\\
&\text{$~~~~~~~=1$ and $j-1\leq l$}\\
\end{cases}
\end{equation}
\begin{equation}
\rho_\omega(i,l,k,j,p,q,r,\alpha,\omega)=
\begin{cases}
\omega&\text{ if $j<2$}\\
\\
\mathcal{E}(\rho_\omega(i,l+1,k,j,p,q,r,\alpha,\omega),\rho_1(i,l+1,k,j,p,q,r,\alpha,\omega),&\text{ } \\ 
\rho_0(i,l+1,k,j,p,q,r,\alpha,\omega),\rho_2(i,l+1,k,j,p,q,r,\alpha,\omega),&\text{ elif $Ip\left(E\left(\frac{k-1+2^{j-1}}{2^{l-1}}-1\right)\right)$}\\ 
\rho_\alpha(i,l+1,k,j,p,q,r,\alpha,\omega),\Theta\left(E\left(\frac{k-1+2^{j-1}-2^l}{2^l}\right),i,p,q,r,\alpha,\omega,0\right))&\text{$~~~~~~~~~~=1$ and $l<j-1$} \\
\\
\Upsilon(\rho_\omega(i,l+1,k,j,p,q,r,\alpha,\omega),\rho_1(i,l+1,k,j,p,q,r,\alpha,\omega),&\text{ } \\ 
\rho_0(i,l+1,k,j,p,q,r,\alpha,\omega),\rho_2(i,l+1,k,j,p,q,r,\alpha,\omega),&\text{ elif $Ip\left(E\left(\frac{k-1+2^{j-1}}{2^{l-1}}-1\right)\right)$}\\ 
\rho_\alpha(i,l+1,k,j,p,q,r,\alpha,\omega),\Theta\left(E\left(\frac{k-1+2^{j-1}-2^l}{2^l}\right),i,p,q,r,\alpha,\omega,0\right))&\text{$~~~~~~~\neq 1$ and $l<j-1$} \\
\\
\mathcal{E}(\omega,p,q,r,\alpha,\Theta\left(E\left(\frac{k-1+2^{j-1}-2^l}{2^l}\right),i,p,q,r,\alpha,\omega,0\right))&\text{ elif $Ip\left(E\left(\frac{k-1+2^{j-1}}{2^{l-1}}-1\right)\right)$}\\
&\text{$~~~~~~~=1$ and $j-1\leq l$}\\
\\
\Upsilon(\omega,p,q,r,\alpha,\Theta\left(E\left(\frac{k-1+2^{j-1}-2^l}{2^l}\right),i,p,q,r,\alpha,\omega,0\right))&\text{ elif $Ip\left(E\left(\frac{k-1+2^{j-1}}{2^{l-1}}-1\right)\right)$}\\
&\text{$~~~~~~~=1$ and $j-1\leq l$}\\
\end{cases}
\end{equation}
\begin{equation}
\aleph_0(i,l,k,p,q,r,\alpha,\omega)=
\begin{cases}
T_0(\aleph_\omega(i,l+1,k,p,q,r,\alpha,\omega),\aleph_1(i,l+1,k,p,q,r,\alpha,\omega),&\text{ } \\ 
\aleph_0(i,l+1,k,p,q,r,\alpha,\omega),\aleph_2(i,l+1,k,p,q,r,\alpha,\omega),&\text{ if $Ip\left(E\left(\frac{k-1}{2^{l-1}}\right)+1\right)$}\\ 
\aleph_\alpha(i,l+1,k,p,q,r,\alpha,\omega),\Theta\left(E\left(\frac{k-1}{2^l}\right)-1+2^{\omega-l},i,p,q,r,\alpha,\omega,0\right))&\text{$~~~~~~~~~~=1$ and $l<\omega$} \\
\\
Z_0(\aleph_\omega(i,l+1,k,p,q,r,\alpha,\omega),\aleph_1(i,l+1,k,p,q,r,\alpha,\omega),&\text{ } \\ 
\aleph_0(i,l+1,k,p,q,r,\alpha,\omega),\aleph_2(i,l+1,k,p,q,r,\alpha,\omega),&\text{ if $Ip\left(E\left(\frac{k-1}{2^{l-1}}\right)+1\right)$}\\ 
\aleph_\alpha(i,l+1,k,p,q,r,\alpha,\omega),\Theta\left(E\left(\frac{k-1}{2^l}\right)-1+2^{\omega-l},i,p,q,r,\alpha,\omega,0\right))&\text{$~~~~~~~\neq 1$ and $l<\omega$} \\
\\
T_0(\omega,p,q,r,\alpha,\Theta\left(E\left(\frac{k-1}{2^l}\right)-1+2^{\omega-l},i,p,q,r,\alpha,\omega,0\right))&\text{ if $Ip\left(E\left(\frac{k-1}{2^{l-1}}\right)+1\right)$}\\
&\text{$~~~~~~~=1$ and $\omega\leq l$}\\
\\
Z_0(\omega,p,q,r,\alpha,\Theta\left(E\left(\frac{k-1}{2^l}\right)-1+2^{\omega-l},i,p,q,r,\alpha,\omega,0\right))&\text{ if $Ip\left(E\left(\frac{k-1}{2^{l-1}}\right)+1\right)$}\\
&\text{$~~~~~~~\neq 1$ and $\omega\leq l$}\\
\end{cases}
\label{Debaleph}
\end{equation}
\begin{equation}
\aleph_1(i,l,k,p,q,r,\alpha,\omega)=
\begin{cases}
T_1(\aleph_\omega(i,l+1,k,p,q,r,\alpha,\omega),\aleph_1(i,l+1,k,p,q,r,\alpha,\omega),&\text{ } \\ 
\aleph_0(i,l+1,k,p,q,r,\alpha,\omega),\aleph_2(i,l+1,k,p,q,r,\alpha,\omega),&\text{ if $Ip\left(E\left(\frac{k-1}{2^{l-1}}\right)+1\right)$}\\ 
\aleph_\alpha(i,l+1,k,p,q,r,\alpha,\omega),\Theta\left(E\left(\frac{k-1}{2^l}\right)-1+2^{\omega-l},i,p,q,r,\alpha,\omega,0\right))&\text{$~~~~~~~~~~=1$ and $l<\omega$} \\
\\
Z_1(\aleph_\omega(i,l+1,k,p,q,r,\alpha,\omega),\aleph_1(i,l+1,k,p,q,r,\alpha,\omega),&\text{ } \\ 
\aleph_0(i,l+1,k,p,q,r,\alpha,\omega),\aleph_2(i,l+1,k,p,q,r,\alpha,\omega),&\text{ if $Ip\left(E\left(\frac{k-1}{2^{l-1}}\right)+1\right)$}\\ 
\aleph_\alpha(i,l+1,k,p,q,r,\alpha,\omega),\Theta\left(E\left(\frac{k-1}{2^l}\right)-1+2^{\omega-l},i,p,q,r,\alpha,\omega,0\right))&\text{$~~~~~~~\neq 1$ and $l<\omega$} \\
\\
T_1(\omega,p,q,r,\alpha,\Theta\left(E\left(\frac{k-1}{2^l}\right)-1+2^{\omega-l},i,p,q,r,\alpha,\omega,0\right))&\text{ if $Ip\left(E\left(\frac{k-1}{2^{l-1}}\right)+1\right)$}\\
&\text{$~~~~~~~=1$ and $\omega\leq l$}\\
\\
Z_1(\omega,p,q,r,\alpha,\Theta\left(E\left(\frac{k-1}{2^l}\right)-1+2^{\omega-l},i,p,q,r,\alpha,\omega,0\right))&\text{ if $Ip\left(E\left(\frac{k-1}{2^{l-1}}\right)+1\right)$}\\
&\text{$~~~~~~~\neq 1$ and $\omega\leq l$}\\
\end{cases}
\end{equation}
\begin{equation}
\aleph_2(i,l,k,p,q,r,\alpha,\omega)=
\begin{cases}
T_2(\aleph_\omega(i,l+1,k,p,q,r,\alpha,\omega),\aleph_1(i,l+1,k,p,q,r,\alpha,\omega),&\text{ } \\ 
\aleph_0(i,l+1,k,p,q,r,\alpha,\omega),\aleph_2(i,l+1,k,p,q,r,\alpha,\omega),&\text{ if $Ip\left(E\left(\frac{k-1}{2^{l-1}}\right)+1\right)$}\\ 
\aleph_\alpha(i,l+1,k,p,q,r,\alpha,\omega),\Theta\left(E\left(\frac{k-1}{2^l}\right)-1+2^{\omega-l},i,p,q,r,\alpha,\omega,0\right))&\text{$~~~~~~~~~~=1$ and $l<\omega$} \\
\\
Z_2(\aleph_\omega(i,l+1,k,p,q,r,\alpha,\omega),\aleph_1(i,l+1,k,p,q,r,\alpha,\omega),&\text{ } \\ 
\aleph_0(i,l+1,k,p,q,r,\alpha,\omega),\aleph_2(i,l+1,k,p,q,r,\alpha,\omega),&\text{ if $Ip\left(E\left(\frac{k-1}{2^{l-1}}\right)+1\right)$}\\ 
\aleph_\alpha(i,l+1,k,p,q,r,\alpha,\omega),\Theta\left(E\left(\frac{k-1}{2^l}\right)-1+2^{\omega-l},i,p,q,r,\alpha,\omega,0\right))&\text{$~~~~~~~\neq 1$ and $l<\omega$} \\
\\
T_2(\omega,p,q,r,\alpha,\Theta\left(E\left(\frac{k-1}{2^l}\right)-1+2^{\omega-l},i,p,q,r,\alpha,\omega,0\right))&\text{ if $Ip\left(E\left(\frac{k-1}{2^{l-1}}\right)+1\right)$}\\
&\text{$~~~~~~~=1$ and $\omega\leq l$}\\
\\
Z_2(\omega,p,q,r,\alpha,\Theta\left(E\left(\frac{k-1}{2^l}\right)-1+2^{\omega-l},i,p,q,r,\alpha,\omega,0\right))&\text{ if $Ip\left(E\left(\frac{k-1}{2^{l-1}}\right)+1\right)$}\\
&\text{$~~~~~~~\neq 1$ and $\omega\leq l$}\\
\end{cases}
\end{equation}
\begin{equation}
\aleph_\alpha(i,l,k,p,q,r,\alpha,\omega)=
\begin{cases}
\beta(\aleph_\omega(i,l+1,k,p,q,r,\alpha,\omega),\aleph_1(i,l+1,k,p,q,r,\alpha,\omega),&\text{ } \\ 
\aleph_0(i,l+1,k,p,q,r,\alpha,\omega),\aleph_2(i,l+1,k,p,q,r,\alpha,\omega),&\text{ if $Ip\left(E\left(\frac{k-1}{2^{l-1}}\right)+1\right)$}\\ 
\aleph_\alpha(i,l+1,k,p,q,r,\alpha,\omega),\Theta\left(E\left(\frac{k-1}{2^l}\right)-1+2^{\omega-l},i,p,q,r,\alpha,\omega,0\right))&\text{$~~~~~~~~~~=1$ and $l<\omega$} \\
\\
\Phi(\aleph_\omega(i,l+1,k,p,q,r,\alpha,\omega),\aleph_1(i,l+1,k,p,q,r,\alpha,\omega),&\text{ } \\ 
\aleph_0(i,l+1,k,p,q,r,\alpha,\omega),\aleph_2(i,l+1,k,p,q,r,\alpha,\omega),&\text{ if $Ip\left(E\left(\frac{k-1}{2^{l-1}}\right)+1\right)$}\\ 
\aleph_\alpha(i,l+1,k,p,q,r,\alpha,\omega),\Theta\left(E\left(\frac{k-1}{2^l}\right)-1+2^{\omega-l},i,p,q,r,\alpha,\omega,0\right))&\text{$~~~~~~~\neq 1$ and $l<\omega$} \\
\\
\beta(\omega,p,q,r,\alpha,\Theta\left(E\left(\frac{k-1}{2^l}\right)-1+2^{\omega-l},i,p,q,r,\alpha,\omega,0\right))&\text{ if $Ip\left(E\left(\frac{k-1}{2^{l-1}}\right)+1\right)$}\\
&\text{$~~~~~~~=1$ and $\omega\leq l$}\\
\\
\Phi(\omega,p,q,r,\alpha,\Theta\left(E\left(\frac{k-1}{2^l}\right)-1+2^{\omega-l},i,p,q,r,\alpha,\omega,0\right))&\text{ if $Ip\left(E\left(\frac{k-1}{2^{l-1}}\right)+1\right)$}\\
&\text{$~~~~~~~\neq 1$ and $\omega\leq l$}\\
\end{cases}
\end{equation}
\begin{equation}
\aleph_\omega(i,l,k,p,q,r,\alpha,\omega)=
\begin{cases}
\mathcal{E}(\aleph_\omega(i,l+1,k,p,q,r,\alpha,\omega),\aleph_1(i,l+1,k,p,q,r,\alpha,\omega),&\text{ } \\ 
\aleph_0(i,l+1,k,p,q,r,\alpha,\omega),\aleph_2(i,l+1,k,p,q,r,\alpha,\omega),&\text{ if $Ip\left(E\left(\frac{k-1}{2^{l-1}}\right)+1\right)$}\\ 
\aleph_\alpha(i,l+1,k,p,q,r,\alpha,\omega),\Theta\left(E\left(\frac{k-1}{2^l}\right)-1+2^{\omega-l},i,p,q,r,\alpha,\omega,0\right))&\text{$~~~~~~~~~~=1$ and $l<\omega$} \\
\\
\Upsilon(\aleph_\omega(i,l+1,k,p,q,r,\alpha,\omega),\aleph_1(i,l+1,k,p,q,r,\alpha,\omega),&\text{ } \\ 
\aleph_0(i,l+1,k,p,q,r,\alpha,\omega),\aleph_2(i,l+1,k,p,q,r,\alpha,\omega),&\text{ if $Ip\left(E\left(\frac{k-1}{2^{l-1}}\right)+1\right)$}\\ 
\aleph_\alpha(i,l+1,k,p,q,r,\alpha,\omega),\Theta\left(E\left(\frac{k-1}{2^l}\right)-1+2^{\omega-l},i,p,q,r,\alpha,\omega,0\right))&\text{$~~~~~~~\neq 1$ and $l<\omega$} \\
\\
\mathcal{E}(\omega,p,q,r,\alpha,\Theta\left(E\left(\frac{k-1}{2^l}\right)-1+2^{\omega-l},i,p,q,r,\alpha,\omega,0\right))&\text{ if $Ip\left(E\left(\frac{k-1}{2^{l-1}}\right)+1\right)$}\\
&\text{$~~~~~~~=1$ and $\omega\leq l$}\\
\\
\Upsilon(\omega,p,q,r,\alpha,\Theta\left(E\left(\frac{k-1}{2^l}\right)-1+2^{\omega-l},i,p,q,r,\alpha,\omega,0\right))&\text{ if $Ip\left(E\left(\frac{k-1}{2^{l-1}}\right)+1\right)$}\\
&\text{$~~~~~~~\neq 1$ and $\omega\leq l$}\\
\end{cases}
\label{Finaleph}
\end{equation}
\newpage

\normalsize
\section{Proof of the theorem 3.3}
This part aims at proving \eqref{ucoef}:
\footnotesize
\begin{multline}
u^\alpha_{\omega,p,q,r}=\sum\limits_{i=0}^{\eta(p,q,r,\alpha,\omega,\omega-1)}\bigg\{ \bigg[\prod\limits_{j=1}^{\omega}\prod\limits_{k=1}^{2^{j-1}}\Gamma(\rho_\omega(i,1,k,j,p,q,r,\alpha,\omega),\rho_1(i,1,k,j,p,q,r,\alpha,\omega),\rho_0(i,1,k,j,p,q,r,\alpha,\omega), \\
 \rho_2(i,1,k,j,p,q,r,\alpha,\omega),\rho_\alpha(i,1,k,j,p,q,r,\alpha,\omega),\Theta(k-2+2^{j-1},i,p,q,r,\alpha,\omega,0))\bigg] \\
\bigg[\prod\limits_{k=1}^{2^\omega}u^{\aleph_\alpha(i,1,k,p,q,r,\alpha,\omega)}_{\aleph_\omega(i,1,k,p,q,r,\alpha,\omega),\aleph_1(i,1,k,p,q,r,\alpha,\omega),\aleph_0(i,1,k,p,q,r,\alpha,\omega),\aleph_2(i,1,k,p,q,r,\alpha,\omega)}\bigg]\bigg\}
\end{multline}
\normalsize
\begin{itemize}
\item First case: $p<2$ or $q<2$ or $r<2$ 
\end{itemize}
By definition of $\eta$ \eqref{eta}, 
\begin{equation}
\eta(p,q,r,\alpha,\omega,\omega-1)=0
\end{equation}
so there is only one term ($i=0$). In the first product
\begin{multline}
\prod\limits_{j=1}^{\omega}\prod\limits_{k=1}^{2^{j-1}}\Gamma(\rho_\omega(i,1,k,j,p,q,r,\alpha,\omega),\rho_1(i,1,k,j,p,q,r,\alpha,\omega),\rho_0(i,1,k,j,p,q,r,\alpha,\omega), \\
 \rho_2(i,1,k,j,p,q,r,\alpha,\omega),\rho_\alpha(i,1,k,j,p,q,r,\alpha,\omega),\Theta(k-2+2^{j-1},i,p,q,r,\alpha,\omega,0))
\end{multline}
all the terms are equal to 1. Indeed, the functions $\rho_x$ are identity functions in this case (for instance $\rho_1$ returns $p$). Then, by definition \eqref{Gamma}, $\Gamma(\omega,p,q,r,\alpha,i)=1$ and the first product is equal to one. The second product
\begin{equation}
\prod\limits_{k=1}^{2^\omega}u^{\aleph_\alpha(i,1,k,p,q,r,\alpha,\omega)}_{\aleph_\omega(i,1,k,p,q,r,\alpha,\omega),\aleph_1(i,1,k,p,q,r,\alpha,\omega),\aleph_0(i,1,k,p,q,r,\alpha,\omega),\aleph_2(i,1,k,p,q,r,\alpha,\omega)}
\end{equation}
corresponds to $u^\alpha_{\omega,p,q,r}$ because only the first term has strictly positive indices. For the sake of understanding, the next table shows the value of $\aleph_\omega(i,1,k,p,q,r,\alpha,\omega)$ for low $\omega$. The results can be transposed to the other functions $\aleph_x$.\newline
\begin{tabular}{cc}
 & $\aleph_\omega(i,1,k,p,q,r,\alpha,\omega)$ \\
$\omega=1,~k=1$& $\mathcal{E}(\omega,p,q,r,\alpha,.)=\omega$\\
$\omega=1,~k=2$& $\Upsilon(\omega,p,q,r,\alpha,.)=-1$\\
$\omega=2,~k=1$& $\mathcal{E}(\mathcal{E}(\omega,p,q,r,\alpha,.),T_1,T_0,T_2,\beta,.)=\omega$ \\
$\omega=2,~k=2$& $\Upsilon(\mathcal{E}(\omega,p,q,r,\alpha,.),T_1,T_0,T_2,\beta,.)=\Upsilon(\omega,p,q,r,\alpha,.)=-1$\\ 
$\omega=2,~k=3$& $\mathcal{E}(\Upsilon(\omega,p,q,r,\alpha,.),Z_1,Z_0,Z_2,\Phi,.)=\mathcal{E}(-1,-1,-1,-1,\alpha,.)=-1$ \\
$\omega=2,~k=4$& $\Upsilon(\Upsilon(\omega,p,q,r,\alpha,.),Z_1,Z_0,Z_2,\Phi,.)=-1$\\  
$\omega=3~...$
\end{tabular}
\newline
$\forall~\omega>0$, the first term of the second product gives $u^\alpha_{\omega,p,q,r}$ because the multiple composition is only composed of functions $\mathcal{E}$. In the others, $\Upsilon$ appears at least once in the composition. Since $\Upsilon(\omega,p,q,r,\alpha,.)=-1$ when at least one of $p,~q,~r$ is inferior to two, the other terms give 1 by the convention $u^\alpha_{-1,p,q,r}=1$.Therefore, the second product yields $u^\alpha_{\omega,p,q,r}$ and \eqref{ucoef} is correct in this case.

\begin{itemize}
\item Second case: $q=2$ and $\alpha=0$ 
\end{itemize}
By definition of $\eta$ \eqref{eta}, 
\begin{equation}
\eta(p,q,r,\alpha,\omega,\omega-1)=1
\end{equation}
The first product can be divided into two parts: $j=1$ and $j>1$. The two terms ($i=0$ and $1$) corresponding to $j=1$ and $k=1$ are equal to
\begin{equation}
\left\lbrace
\begin{array}{lc}
\Gamma(\omega,p,q,r,\alpha,0)=-\frac{p+1}{q}\\
\Gamma(\omega,p,q,r,\alpha,1)=-\frac{r+1}{q}
\end{array}\right.
\end{equation}
since $j<2$. Now, if $j>1$, necessarily
\footnotesize
\begin{multline}
\Gamma(\rho_\omega(i,1,k,j,p,q,r,\alpha,\omega),\rho_1(i,1,k,j,p,q,r,\alpha,\omega),\rho_0(i,1,k,j,p,q,r,\alpha,\omega),\rho_2(i,1,k,j,p,q,r,\alpha,\omega), \\
\rho_\alpha(i,1,k,j,p,q,r,\alpha,\omega),\Theta(k-2+2^{j-1},i,p,q,r,\alpha,\omega,0))=\Gamma(\omega',p',q',r',\alpha',i)=1
\end{multline}
\normalsize
because $q'<2$ in such a case. So the first product is equal to $-\frac{p+1}{q}$ if $i=0$ and $-\frac{r+1}{q}$ if $i=1$. \newline
The considerations of the first item about the second product can be applied in this case and only the first term of the product is different from one. According to the value of $i$, the first term yields:
\begin{equation}
\left\lbrace
\begin{array}{lc}
u^{\aleph_\alpha(0,1,0,p,q,r,\alpha,\omega)}_{\aleph_\omega(0,1,0,p,q,r,\alpha,\omega),\aleph_1(0,1,0,p,q,r,\alpha,\omega),\aleph_0(0,1,0,p,q,r,\alpha,\omega),\aleph_2(0,1,0,p,q,r,\alpha,\omega)}=u^1_{\omega,p+1,q-1,r} \\
u^{\aleph_\alpha(1,1,0,p,q,r,\alpha,\omega)}_{\aleph_\omega(1,1,0,p,q,r,\alpha,\omega),\aleph_1(1,1,0,p,q,r,\alpha,\omega),\aleph_0(1,1,0,p,q,r,\alpha,\omega),\aleph_2(1,1,0,p,q,r,\alpha,\omega)}=u^2_{\omega,p,q-1,r+1}
\end{array}\right.
\end{equation}
Then,
\begin{equation}
u^0_{\omega,p,2,r}=-\frac{p+1}{q}u^1_{\omega,p+1,q-1,r}-\frac{r+1}{q}u^2_{\omega,p,q-1,r+1}
\end{equation}
which is the expected result \eqref{DisDiv}.
\begin{itemize}
\item General case $\omega>0,~p>1,~q>1,~r>1$ and $(\alpha,q)\neq(0,2)$
\end{itemize}
The proof is made by mathematical induction on $\omega$.\newline
\textbf{Base case $\omega=1$}\\
The proposition 3.2 \eqref{step3} yields
\begin{multline}
u^\alpha_{1,p,q,r}=\sum_{i=0}^{\mathcal{F}(p,q,r,\alpha,1)}\Gamma(1,p,q,r,\alpha,i)
u^{\beta(1,p,q,r,\alpha,i)}_{\mathcal{E}(1,p,q,r,\alpha,i),T_1(1,p,q,r,\alpha,i),T_0(1,p,q,r,\alpha,i),T_2(1,p,q,r,\alpha,i)}\\
u^{\Phi(1,p,q,r,\alpha,i)}_{\Upsilon(1,p,q,r,\alpha,i),Z_1(1,p,q,r,\alpha,i),Z_0(1,p,q,r,\alpha,i),Z_2(1,p,q,r,\alpha,i)}
\label{correct}
\end{multline}
According to the equation \eqref{ucoef}, $u^\alpha_{1,p,q,r}$ is equal to
\footnotesize
\begin{multline}
\sum\limits_{i=0}^{\eta(p,q,r,\alpha,1,0)}\bigg\{ \bigg[\prod\limits_{j=1}^{1}\prod\limits_{k=1}^{2^{j-1}}\Gamma(\rho_\omega(i,1,k,j,p,q,r,\alpha,1),\rho_1(i,1,k,j,p,q,r,\alpha,1),\\ \rho_0(i,1,k,j,p,q,r,\alpha,1),\rho_2(i,1,k,j,p,q,r,\alpha,1),\rho_\alpha(i,1,k,j,p,q,r,\alpha,1),\\ \Theta(k-2+2^{j-1},i,p,q,r,\alpha,1,0))\bigg] 
\bigg[\prod\limits_{k=1}^{2}u^{\aleph_\alpha(i,1,k,p,q,r,\alpha,1)}_{\aleph_\omega(i,1,k,p,q,r,\alpha,1),\aleph_1(i,1,k,p,q,r,\alpha,1),\aleph_0(i,1,k,p,q,r,\alpha,1),\aleph_2(i,1,k,p,q,r,\alpha,1)}\bigg]\bigg\}
\label{uuu}
\end{multline}
\normalsize
In the first product, $j=1$ so the functions $\rho_x$ gives by definition $p,q,r,\alpha,\omega$. Moreover, since $k=1$, $k-2+2^{j-1}=0$ so
\begin{equation}
\Theta(k-2+2^{j-1},i,p,q,r,\alpha,1,0)=\theta(i,p,q,r,\alpha,1,0,0)=i
\end{equation}
Consequently
\begin{multline}
\bigg[\prod\limits_{j=1}^{1}\prod\limits_{k=1}^{2^{j-1}}\Gamma(\rho_\omega(i,1,k,j,p,q,r,\alpha,1),\rho_1(i,1,k,j,p,q,r,\alpha,1),\rho_0(i,1,k,j,p,q,r,\alpha,1), \\
 \rho_2(i,1,k,j,p,q,r,\alpha,1),\rho_\alpha(i,1,k,j,p,q,r,\alpha,1),\Theta(k-2+2^{j-1},i,p,q,r,\alpha,1,0))\bigg]\\
 =\Gamma(1,p,q,r,\alpha,i)
\end{multline}
In the second product, there are two terms. The functions from (A.68) to (A.72) require the computation of $I_p\left(E\left(\frac{k-1}{2^{l-1}}\right)+1\right)$ only for $l=1$. Then,
\begin{equation}
I_p(E(k-1)+1)=I_p(k)=
\left\lbrace
\begin{array}{lc}
2\text{ if $k=0$}\\
1\text{ if $k=1$}
\end{array}\right.
\end{equation}
and
\begin{equation}
\left\lbrace
\begin{array}{lc}
\aleph_1(i,1,0,p,q,r,\alpha,1)=Z_1(1,p,q,r,\alpha,\Theta(0,i,p,q,r,\alpha,1,0))=Z_1(1,p,q,r,\alpha,i) \\
\aleph_1(i,1,k,p,q,r,\alpha,1)=T_1(1,p,q,r,\alpha,\Theta(0,i,p,q,r,\alpha,1,0))=T_1(1,p,q,r,\alpha,i)
\end{array}\right.
\end{equation}
The other functions $\aleph_x$ provide similar results. Then, the second product yields
\begin{multline}
\bigg[\prod\limits_{k=1}^{2}u^{\aleph_\alpha(i,1,k,p,q,r,\alpha,1)}_{\aleph_\omega(i,1,k,p,q,r,\alpha,1),\aleph_1(i,1,k,p,q,r,\alpha,1),\aleph_0(i,1,k,p,q,r,\alpha,1),\aleph_2(i,1,k,p,q,r,\alpha,1)}\bigg]=\\
u^{\beta(1,p,q,r,\alpha,i)}_{\mathcal{E}(1,p,q,r,\alpha,i),T_1(1,p,q,r,\alpha,i),T_0(1,p,q,r,\alpha,i),T_2(1,p,q,r,\alpha,i)}
u^{\Phi(1,p,q,r,\alpha,i)}_{\Upsilon(1,p,q,r,\alpha,i),Z_1(1,p,q,r,\alpha,i),Z_0(1,p,q,r,\alpha,i),Z_2(1,p,q,r,\alpha,i)}
\end{multline}
Therefore, the equation \eqref{uuu} becomes
\begin{multline}
\sum\limits_{i=0}^{\eta(p,q,r,\alpha,1,0)}\Gamma(1,p,q,r,\alpha,i)
u^{\beta(1,p,q,r,\alpha,i)}_{\mathcal{E}(1,p,q,r,\alpha,i),T_1(1,p,q,r,\alpha,i),T_0(1,p,q,r,\alpha,i),T_2(1,p,q,r,\alpha,i)}\\
u^{\Phi(1,p,q,r,\alpha,i)}_{\Upsilon(1,p,q,r,\alpha,i),Z_1(1,p,q,r,\alpha,i),Z_0(1,p,q,r,\alpha,i),Z_2(1,p,q,r,\alpha,i)}
\label{Here}
\end{multline}
By definition of $\eta$, $\eta(p,q,r,\alpha,1,0)=\mathcal{F}(p,q,r,\alpha,1)$ and \eqref{Here} is then identical to \eqref{correct}. The formula \eqref{ucoef} holds for $\omega=1$.\\
\textbf{Induction step}\\
Now, the relation \eqref{ucoef} is assumed to hold for the integers $\omega_0$ such as $\omega_0<\omega$. Is \eqref{ucoef} still true for $\omega$ ? \\
The proposition 3.2 \eqref{step3} yields
\begin{multline}
u^\alpha_{\omega,p,q,r}=\sum_{i=0}^{\mathcal{F}(p,q,r,\alpha,\omega)}\Gamma(\omega,p,q,r,\alpha,i)
u^{\beta(\omega,p,q,r,\alpha,i)}_{\mathcal{E}(\omega,p,q,r,\alpha,i),T_1(\omega,p,q,r,\alpha,i),T_0(\omega,p,q,r,\alpha,i),T_2(\omega,p,q,r,\alpha,i)}\\
u^{\Phi(\omega,p,q,r,\alpha,i)}_{\Upsilon(\omega,p,q,r,\alpha,i),Z_1(\omega,p,q,r,\alpha,i),Z_0(\omega,p,q,r,\alpha,i),Z_2(\omega,p,q,r,\alpha,i)}
\label{w_1}
\end{multline}
In the case $p,q,r>1$ and $(\alpha,q)\neq (0,2)$, $\mathcal{E}(\omega,p,q,r,\alpha,i)<\omega$ and $\Upsilon(\omega,p,q,r,\alpha,i)<\omega$ so the induction hypothesis can be applied to \eqref{w_1}. 
\newline
Before proceeding, additional information about the construction of the solution \eqref{ucoef} must be provided. In the equation \eqref{uu}, $\widetilde{\omega}$ is equal to $\omega$ to be in agreement with the solution \eqref{ucoef}.
\footnotesize
\begin{multline}
u^\alpha_{\omega,p,q,r}=\sum\limits_{i=0}^{\eta(p,q,r,\alpha,\omega,\widetilde{\omega}-1)}\bigg\{ \bigg[\prod\limits_{j=1}^{\widetilde{\omega}}\prod\limits_{k=1}^{2^{j-1}}\Gamma(\rho_\omega(i,1,k,j,p,q,r,\alpha,\omega),\rho_1(i,1,k,j,p,q,r,\alpha,\omega),\rho_0(i,1,k,j,p,q,r,\alpha,\omega), \\
 \rho_2(i,1,k,j,p,q,r,\alpha,\omega),\rho_\alpha(i,1,k,j,p,q,r,\alpha,\omega),\Theta(k-2+2^{j-1},i,p,q,r,\alpha,\omega,0))\bigg] \\
\bigg[\prod\limits_{k=1}^{2^{\widetilde{\omega}}}u^{\aleph_\alpha(i,1,k,p,q,r,\alpha,\omega)}_{\aleph_\omega(i,1,k,p,q,r,\alpha,\omega),\aleph_1(i,1,k,p,q,r,\alpha,\omega),\aleph_0(i,1,k,p,q,r,\alpha,\omega),\aleph_2(i,1,k,p,q,r,\alpha,\omega)}\bigg]\bigg\}
\label{uu}
\end{multline}
\normalsize
However, the solution has been built so that $\widetilde{\omega}$ is not every time defined by $\widetilde{\omega}=\omega$. Otherwise, the calculation of $u^{\beta(\omega,p,q,r,\alpha,i)}_{\mathcal{E}(\omega,p,q,r,\alpha,i),T_1(\omega,p,q,r,\alpha,i),T_0(\omega,p,q,r,\alpha,i),T_2(\omega,p,q,r,\alpha,i)}$ would require $\widetilde{\omega}=\mathcal{E}(\omega,p,q,r,\alpha,i)$ which would complicate a lot the expression of the analytical solution \eqref{ucoef}. Instead, $\widetilde{\omega}$ is equal to $\omega-1$ in this case even if this formulation adds some zero terms since $\mathcal{E}(\omega,p,q,r,\alpha,i)$ can be inferior to $\omega-1$. So, the definition of $\widetilde{\omega}$ is $\displaystyle \max_{i}~f(\omega,p,q,r,\alpha,i)$ where $f$ is defined by $u^\alpha_{f(\omega,p,q,r,\alpha,i),p,q,r}$. For $u^\alpha_{\omega,p,q,r}$, $f=id$ and $\widetilde{\omega}=\omega$ as expected in \eqref{ucoef}. \newline
Similarly, the functions $\aleph$ (from \eqref{Debaleph} to \eqref{Finaleph}) and $\Theta$ are affected by the definition of $\widetilde{\omega}$. For instance, the function $\aleph_0$ is defined by
\footnotesize
\begin{equation}
\aleph_0(i,l,k,p,q,r,\alpha,\omega)=
\begin{cases}
T_0(\aleph_\omega(i,l+1,k,p,q,r,\alpha,\omega),\aleph_1(i,l+1,k,p,q,r,\alpha,\omega),&\text{ } \\ 
\aleph_0(i,l+1,k,p,q,r,\alpha,\omega),\aleph_2(i,l+1,k,p,q,r,\alpha,\omega),&\text{ if $Ip\left(E\left(\frac{k-1}{2^{l-1}}\right)+1\right)$}\\ 
\aleph_\alpha(i,l+1,k,p,q,r,\alpha,\omega),\Theta\left(E\left(\frac{k-1}{2^l}\right)-1+2^{\widetilde{\omega}-l},i,p,q,r,\alpha,\omega,0\right))&\text{$~~~~~~~=1$ and $l<\widetilde{\omega}$} \\
\\
Z_0(\aleph_\omega(i,l+1,k,p,q,r,\alpha,\omega),\aleph_1(i,l+1,k,p,q,r,\alpha,\omega),&\text{ } \\ 
\aleph_0(i,l+1,k,p,q,r,\alpha,\omega),\aleph_2(i,l+1,k,p,q,r,\alpha,\omega),&\text{ if $Ip\left(E\left(\frac{k-1}{2^{l-1}}\right)+1\right)$}\\ 
\aleph_\alpha(i,l+1,k,p,q,r,\alpha,\omega),\Theta\left(E\left(\frac{k-1}{2^l}\right)-1+2^{\widetilde{\omega}-l},i,p,q,r,\alpha,\omega,0\right))&\text{$~~~~~~~\neq 1$ and $l<\widetilde{\omega}$} \\
\\
T_0(\omega,p,q,r,\alpha,\Theta\left(E\left(\frac{k-1}{2^l}\right)-1+2^{\widetilde{\omega}-l},i,p,q,r,\alpha,\omega,0\right))&\text{ if $Ip\left(E\left(\frac{k-1}{2^{l-1}}\right)+1\right)$}\\
&\text{$~~~~~~~=1$ and $\widetilde{\omega}\leq l$}\\
\\
Z_0(\omega,p,q,r,\alpha,\Theta\left(E\left(\frac{k-1}{2^l}\right)-1+2^{\widetilde{\omega}-l},i,p,q,r,\alpha,\omega,0\right))&\text{ if $Ip\left(E\left(\frac{k-1}{2^{l-1}}\right)+1\right)$}\\
&\text{$~~~~~~~\neq 1$ and $\widetilde{\omega}\leq l$}\\
\end{cases}
\end{equation} 
\normalsize
The functions $\aleph_1$, $\aleph_2$, $\aleph_\alpha$ and $\aleph_\omega$ are similarly defined. Moreover, $\Theta$ is defined by
\scriptsize
\begin{equation}
\Theta(m,i,p,q,r,\alpha,\omega,h)=
\begin{cases}
\theta(i,p,q,r,\alpha,\omega,Ip_b(h,m),\widetilde{\omega}-1)&\text{ if $h=E\left(\frac{ln(m+1)}{ln(2)}\right)$}\\
\theta(\Theta(m,i,p,q,r,\alpha,\omega,h+1),C_1(m,h,0,i,p,q,r,\alpha,\omega),\\ C_0(m,h,0,i,p,q,r,\alpha,\omega),
C_2(m,h,0,i,p,q,r,\alpha,\omega),C_\alpha(m,h,0,i,p,q,r,\alpha,\omega),&\text{ otherwise}\\ C_\omega(m,h,0,i,p,q,r,\alpha,\omega),
Ip_b(h,m),\widetilde{\omega}-1+h-E\left(\frac{ln(m+1)}{ln(2)}\right))
\end{cases}
\end{equation}
\normalsize
Note that this discussion does not regard the calculation of $u^\alpha_{\omega,p,q,r}$.
With these considerations, the introduction of the induction hypothesis in \eqref{w_1} yields 
\tiny
\begin{multline}
u^\alpha_{\omega,p,q,r}=\sum_{i=0}^{\mathcal{F}(p,q,r,\alpha,\omega)}\Gamma(\omega,p,q,r,\alpha,i)\bigg[\sum\limits_{i_1=0}^{\eta(T_1,T_0,T_2,\beta,\mathcal{E},\omega-2)}
\bigg\{ \bigg[\prod\limits_{j_1=1}^{\omega-1}\prod\limits_{k_1=1}^{2^{j_1-1}}\Gamma(\rho_\omega(i_1,1,k_1,j_1,T_1,T_0,T_2,\beta,\mathcal{E}),\rho_1(i_1,1,k_1,j_1,T_1,T_0,T_2,\beta,\mathcal{E}),\\
\rho_0(i_1,1,k_1,j_1,T_1,T_0,T_2,\beta,\mathcal{E}),\rho_2(i_1,1,k_1,j_1,T_1,T_0,T_2,\beta,\mathcal{E}),\rho_\alpha(i_1,1,k_1,j_1,T_1,T_0,T_2,\beta,\mathcal{E}), 
\Theta(k_1-2+2^{j_1-1},i_1,T_1,T_0,T_2,\beta,\mathcal{E},0))\bigg] \\
\bigg[\prod\limits_{k'_1=1}^{2^{\omega-1}}u^{\aleph_\alpha(i_1,1,k'_1,T_1,T_0,T_2,\beta,\mathcal{E})}_{\aleph_\omega(i_1,1,k'_1,T_1,T_0,T_2,\beta,\mathcal{E}),
\aleph_1(i_1,1,k'_1,T_1,T_0,T_2,\beta,\mathcal{E}),\aleph_0(i_1,1,k'_1,T_1,T_0,T_2,\beta,\mathcal{E}),\aleph_2(i_1,1,k'_1,T_1,T_0,T_2,\beta,\mathcal{E})}\bigg]\bigg\}\bigg] \\
\bigg[\sum\limits_{i_2=0}^{\eta(Z_1,Z_0,Z_2,\Phi,\Upsilon,\omega-2)}\bigg\{ \bigg[\prod\limits_{j_2=1}^{\omega-1}\prod\limits_{k_2=1}^{2^{j_2-1}}
\Gamma(\rho_\omega(i_2,1,k_2,j_2,Z_1,Z_0,Z_2,\Phi,\Upsilon),\rho_1(i_2,1,k_2,j_2,Z_1,Z_0,Z_2,\Phi,\Upsilon),\rho_0(i_2,1,k_2,j_2,Z_1,Z_0,Z_2,\Phi,\Upsilon), \\
 \rho_2(i_2,1,k_2,j_2,Z_1,Z_0,Z_2,\Phi,\Upsilon),\rho_\alpha(i_2,1,k_2,j_2,Z_1,Z_0,Z_2,\Phi,\Upsilon),\Theta(k_2-2+2^{j_2-1},i_2,Z_1,Z_0,Z_2,\Phi,\Upsilon,0))\bigg] \\
\bigg[\prod\limits_{k'_2=1}^{2^{\omega-1}}u^{\aleph_\alpha(i_2,1,k'_2,Z_1,Z_0,Z_2,\Phi,\Upsilon)}_{\aleph_\omega(i_2,1,k'_2,Z_1,Z_0,Z_2,\Phi,\Upsilon),
\aleph_1(i_2,1,k'_2,Z_1,Z_0,Z_2,\Phi,\Upsilon),\aleph_0(i_2,1,k'_2,Z_1,Z_0,Z_2,\Phi,\Upsilon),\aleph_2(i_2,1,k'_2,Z_1,Z_0,Z_2,\Phi,\Upsilon)}\bigg]\bigg\}\bigg]
\end{multline}
\normalsize
For the sake of shortness, the arguments of the functions $T_1,~T_0,~T_2,~\beta,~\mathcal{E},~Z_1,~Z_0,~Z_2,~\Phi~\text{and }\Upsilon$ are not written (for instance $T_1=T_1(\omega,p,q,r,\alpha,i)$).\newline
The next step is to merge the three sums into one. For the sake of simplification, the following notations are introduced
\begin{equation}
\left\lbrace
\begin{array}{lc}
L=\mathcal{F}(p,q,r,\alpha,\omega) \\
M(i)=\eta(T_1,T_0,T_2,\beta,\mathcal{E},\omega-2) \\
N(i)=\eta(Z_1,Z_0,Z_2,\Phi,\Upsilon,\omega-2)
\end{array}\right.
\end{equation}
By definition
\begin{equation}
\begin{array}{lc}
M(i)=\widetilde{\eta_1}(p,q,r,\alpha,\omega,i,\omega-2)\\
N(i)=\widetilde{\eta_2}(p,q,r,\alpha,\omega,i,\omega-2)
\end{array}
\end{equation}
The merger is then given by 
\begin{equation}
\sum\limits_{i=0}^{L}\sum\limits_{i_1=0}^{M(i)}\sum\limits_{i_2=0}^{N(i)} \rightarrow \sum\limits_{n=0}^{P}
\end{equation}
So,
\begin{equation}
P+1=\sum\limits_{i=0}^{L}(M(i)+1)(N(i)+1)
\end{equation}
By definition of $\eta$ \eqref{eta},
\begin{equation}
P=\eta(p,q,r,\alpha,\omega,\omega-1)
\end{equation}
Moreover, $i,~i_1$ and $i_2$ must be expressed as functions of $n$. Schematically, 
\begin{equation}
\begin{array}{lccc}
i=0&i_1=0&i_2=0&n=0 \\
0&0&1&1\\
.&.&.&.\\
.&.&.&.\\
0&0&N(0)&N(0)\\
0&1&0&N(0)+1\\
.&.&.&.\\
.&.&.&.\\
0&M(0)&N(0)&(M(0)+1)(N(0)+1)-1\\
1&0&0&(M(0)+1)(N(0)+1)\\
...&...&...&...
\end{array}
\end{equation}
Mathematically, $i$ is computed by:
\begin{equation}
i=i\text{ such as }n\in\left[\sum\limits_{k=0}^{i-1}(M(k)+1)(N(k)+1);\sum\limits_{k=0}^{i}(M(k)+1)(N(k)+1)\right[
\end{equation}
$i_1$ is the quotient of an Euclidean division:
\begin{equation}
i_1=E\left(\frac{n-\sum\limits_{m=0}^{i-1}(M(m)+1)(N(m)+1)}{N(i)+1}\right)
\end{equation}
$i_2$ is the remainder of an Euclidean division \eqref{R}:
\begin{equation}
i_2=\mathcal{R}\left(n-\sum\limits_{m=0}^{i-1}(M(m)+1)(N(m)+1),N(i)+1\right)
\end{equation}
Then, $i,~i_1$ and $i_2$ correspond to
\begin{equation}
\begin{array}{lc}
i=\theta(n,p,q,r,\alpha,\omega,0,\omega-1) \\
i_1=\theta(n,p,q,r,\alpha,\omega,1,\omega-1) \\
i_2=\theta(n,p,q,r,\alpha,\omega,2,\omega-1) \\
\end{array}
\end{equation}
In the end, the merge into one sum leads to
\scriptsize
\begin{multline}
u^\alpha_{\omega,p,q,r}=\sum_{n=0}^{\eta(p,q,r,\alpha,\omega,\omega-1)}\Gamma(\omega,p,q,r,\alpha,i)
\bigg\{ \bigg[\prod\limits_{j_1=1}^{\omega-1}\prod\limits_{k_1=1}^{2^{j_1-1}}\Gamma(\rho_\omega(i_1,1,k_1,j_1,T_1,T_0,T_2,\beta,\mathcal{E}),\rho_1(i_1,1,k_1,j_1,T_1,T_0,T_2,\beta,\mathcal{E}),\\
\rho_0(i_1,1,k_1,j_1,T_1,T_0,T_2,\beta,\mathcal{E}),\rho_2(i_1,1,k_1,j_1,T_1,T_0,T_2,\beta,\mathcal{E}),\rho_\alpha(i_1,1,k_1,j_1,T_1,T_0,T_2,\beta,\mathcal{E}), 
\Theta(k_1-2+2^{j_1-1},i_1,T_1,T_0,T_2,\beta,\mathcal{E},0))\bigg] \\
\bigg[\prod\limits_{k'_1=1}^{2^{\omega-1}}u^{\aleph_\alpha(i_1,1,k'_1,T_1,T_0,T_2,\beta,\mathcal{E})}_{\aleph_\omega(i_1,1,k'_1,T_1,T_0,T_2,\beta,\mathcal{E}),
\aleph_1(i_1,1,k'_1,T_1,T_0,T_2,\beta,\mathcal{E}),\aleph_0(i_1,1,k'_1,T_1,T_0,T_2,\beta,\mathcal{E}),\aleph_2(i_1,1,k'_1,T_1,T_0,T_2,\beta,\mathcal{E})}\bigg]\bigg\} \\
\bigg\{ \bigg[\prod\limits_{j_2=1}^{\omega-1}\prod\limits_{k_2=1}^{2^{j_2-1}}
\Gamma(\rho_\omega(i_2,1,k_2,j_2,Z_1,Z_0,Z_2,\Phi,\Upsilon),\rho_1(i_2,1,k_2,j_2,Z_1,Z_0,Z_2,\Phi,\Upsilon),\rho_0(i_2,1,k_2,j_2,Z_1,Z_0,Z_2,\Phi,\Upsilon), \\
 \rho_2(i_2,1,k_2,j_2,Z_1,Z_0,Z_2,\Phi,\Upsilon),\rho_\alpha(i_2,1,k_2,j_2,Z_1,Z_0,Z_2,\Phi,\Upsilon),\Theta(k_2-2+2^{j_2-1},i_2,Z_1,Z_0,Z_2,\Phi,\Upsilon,0))\bigg] \\
\bigg[\prod\limits_{k'_2=1}^{2^{\omega-1}}u^{\aleph_\alpha(i_2,1,k'_2,Z_1,Z_0,Z_2,\Phi,\Upsilon)}_{\aleph_\omega(i_2,1,k'_2,Z_1,Z_0,Z_2,\Phi,\Upsilon),
\aleph_1(i_2,1,k'_2,Z_1,Z_0,Z_2,\Phi,\Upsilon),\aleph_0(i_2,1,k'_2,Z_1,Z_0,Z_2,\Phi,\Upsilon),\aleph_2(i_2,1,k'_2,Z_1,Z_0,Z_2,\Phi,\Upsilon)}\bigg]\bigg\}
\end{multline}
\normalsize
Two equalities remain to prove. First,
\begin{multline}
\bigg[\prod\limits_{k'_1=1}^{2^{\omega-1}}u^{\aleph_\alpha(i_1,1,k'_1,T_1,T_0,T_2,\beta,\mathcal{E})}_{\aleph_\omega(i_1,1,k'_1,T_1,T_0,T_2,\beta,\mathcal{E}),
\aleph_1(i_1,1,k'_1,T_1,T_0,T_2,\beta,\mathcal{E}),\aleph_0(i_1,1,k'_1,T_1,T_0,T_2,\beta,\mathcal{E}),\aleph_2(i_1,1,k'_1,T_1,T_0,T_2,\beta,\mathcal{E})}\bigg] \\
\bigg[\prod\limits_{k'_2=1}^{2^{\omega-1}}u^{\aleph_\alpha(i_2,1,k'_2,Z_1,Z_0,Z_2,\Phi,\Upsilon)}_{\aleph_\omega(i_2,1,k'_2,Z_1,Z_0,Z_2,\Phi,\Upsilon),
\aleph_1(i_2,1,k'_2,Z_1,Z_0,Z_2,\Phi,\Upsilon),\aleph_0(i_2,1,k'_2,Z_1,Z_0,Z_2,\Phi,\Upsilon),\aleph_2(i_2,1,k'_2,Z_1,Z_0,Z_2,\Phi,\Upsilon)}\bigg] \\
\stackrel{?}{=} \bigg[\prod\limits_{k=1}^{2^\omega}u^{\aleph_\alpha(n,1,k,p,q,r,\alpha,\omega)}_{\aleph_\omega(n,1,k,p,q,r,\alpha,\omega),\aleph_1(n,1,k,p,q,r,\alpha,\omega),\aleph_0(n,1,k,p,q,r,\alpha,\omega),\aleph_2(n,1,k,p,q,r,\alpha,\omega)}\bigg]
\label{FirstEquality}
\end{multline}
The left member corresponds to a single product with the same size as the right member ($2^{\omega-1}+2^{\omega-1}=2\times 2^{\omega-1}=2^\omega$).\newline
To prove the previous relation \eqref{FirstEquality}, an essential relation about $\Theta$ is needed. From the definitions of $\Theta$ and the functions $C_x$ (from \eqref{DebC} to \eqref{FinC}), formal calculations provide two relations
\scriptsize
\begin{equation}
\Theta\left(E\left(\frac{k-1}{2^l}\right)-1+2^{y-l},n,p,q,r,\alpha,\omega,0\right)=
\left\lbrace
\begin{array}{lc}
\Theta\left(E\left(\frac{k-1}{2^l}\right)-1+2^{y-l-1},i_1,T_1,T_0,T_2,\beta,\mathcal{E},0\right)&\text{ if $k\leq 2^{y-1}$}\\
\Theta\left(E\left(\frac{k-2^{y-1}-1}{2^l}\right)-1+2^{y-l-1},i_2,Z_1,Z_0,Z_2,\Phi,\Upsilon,0\right)&\text{ if $2^{y-1}<k$}\\
\end{array}\right.
\label{Theta_rel}
\end{equation}
\begin{multline}
C_x\left(E\left(\frac{k-1}{2^l}\right)-1+2^{y-l},E\left(\frac{ln\left(E\left(\frac{k-1}{2^l}\right)+2^{y-l}\right)}{ln(2)}\right)-2,0,n,p,q,r,\alpha,\omega\right)=\\
\left\lbrace
\begin{array}{lc}
C_x\left(E\left(\frac{k-1}{2^l}\right)-1+2^{y-1-l},E\left(\frac{ln\left(E\left(\frac{k-1}{2^l}\right)+2^{y-1-l}\right)}{ln(2)}\right)-1,0,i_1,T_1,T_0,T_2,\beta,\mathcal{E},0\right)&\text{ if $k\leq 2^{y-1}$}\\
C_x\left(E\left(\frac{k-2^{y-1}-1}{2^l}\right)-1+2^{y-1-l},E\left(\frac{ln\left(E\left(\frac{k-2^{y-1}-1}{2^l}\right)+2^{y-1-l}\right)}{ln(2)}\right)-1,0,i_2,Z_1,Z_0,Z_2,\Phi,\Upsilon,0\right)&\text{ if $2^{y-1}<k$}\\
\end{array}\right.
\end{multline}
\normalsize
while $l<y$ and $k\leq 2^y$.\newline
With $y=\widetilde{\omega}$, \eqref{Theta_rel} implies the following relation
\footnotesize
\begin{equation}
\aleph_x\left(n,1,k,p,q,r,\alpha,\omega\right)=
\left\lbrace
\begin{array}{lc}
\aleph_x\left(i_1,1,k_1,T_1,T_0,T_2,\beta,\mathcal{E}\right)&\text{ if $k\leq 2^{\omega-1}$}\\
\aleph_x\left(i_2,1,k_2,Z_1,Z_0,Z_2,\Phi,\Upsilon,0\right)&\text{ if $2^{\omega-1}<k$}\\
\end{array}\right.
\label{Sol1}
\end{equation}
\normalsize
with $k_1=k$ and $k_2=k-2^{\omega-1}$. The proof of \eqref{Sol1} is only presented for the case $x=\omega$ because the others are similarly proven:
\begin{multline}
\aleph_\omega(n,1,k,p,q,r,\alpha,\omega)=\stackrel{\mathcal{E}}{\Upsilon}\bigg(\stackrel{\mathcal{E}}{\Upsilon}\bigg(...\bigg(\stackrel{\mathcal{E}}{\Upsilon}\bigg(\stackrel{\mathcal{E}}{\Upsilon}\bigg(\omega,\Theta\bigg(E\left(\frac{k-1}{2^{\omega}}\right),n,p,q,r,\alpha,\omega,0\bigg)\bigg),\\ \Theta\bigg(E\left(\frac{k-1}{2^{\omega-1}}\right)+1,n,p,q,r,\alpha,\omega,0\bigg)\bigg),...\bigg), \Theta\bigg(E\left(\frac{k-1}{4}\right)+2^{\omega-2}-1,n,p,q,r,\alpha,\omega,0\bigg)\bigg),\\
\Theta\bigg(E\left(\frac{k-1}{2}\right)+2^{\omega-1}-1,n,p,q,r,\alpha,\omega,0\bigg)\bigg)
\label{Int1}
\end{multline}
$\mathcal{E}(p,q,r,\alpha,\omega,i)$ has been abbreviated as $\mathcal{E}(\omega,i)$ to alleviate the notations. The notation $\stackrel{\mathcal{E}}{\Upsilon}$ means:
\begin{equation}
\stackrel{\mathcal{E}}{\Upsilon}=
\left\lbrace\begin{array}{lc}
\mathcal{E} & \text{ if $Ip\left(E\left(\frac{k-1}{2^{l-1}}\right)+1\right)=1$} \\
\Upsilon & \text{ if $Ip\left(E\left(\frac{k-1}{2^{l-1}}\right)+1\right)\neq 1$} \\
\end{array}\right.
\end{equation}
Since $k\leq 2^{\omega}$,
\begin{equation}
\begin{aligned}
E\left(\frac{k-1}{2^{\omega}}\right)=0\Rightarrow\Theta\bigg(E\left(\frac{k-1}{2^{\omega}}\right),n,p,q,r,\alpha,\omega,0\bigg)&=\Theta(0,n,p,q,r,\alpha,\omega,0)\\
&=\theta(n,p,q,r,\alpha,\omega,0,\omega-1)\\
&=i
\end{aligned}
\end{equation}
So, the relation \eqref{Int1} becomes
\small
\begin{multline}
\aleph_\omega(n,1,k,p,q,r,\alpha,\omega)=\stackrel{\mathcal{E}}{\Upsilon}\bigg(\stackrel{\mathcal{E}}{\Upsilon}\bigg(...\bigg(\stackrel{\mathcal{E}}{\Upsilon}\bigg(\stackrel{\mathcal{E}}{\Upsilon}(\omega,i), \Theta\bigg(E\left(\frac{k-1}{2^{\omega-1}}\right)+1,n,p,q,r,\alpha,\omega,0\bigg)\bigg),...\bigg),\\ \Theta\bigg(E\left(\frac{k-1}{4}\right)+2^{\omega-2}-1,n,p,q,r,\alpha,\omega,0\bigg)\bigg),\Theta\bigg(E\left(\frac{k-1}{2}\right)+2^{\omega-1}-1,n,p,q,r,\alpha,\omega,0\bigg)\bigg)
\end{multline}
\normalsize
The other part of the equality \eqref{Sol1} is similarly expanded with the notations $k_{1,2}$ and $i_{1,2}$ for $k_1$ or $k_2$ and $i_1$ or $i_2$ depending on the case:
\footnotesize
\begin{multline}
\aleph_\omega\left(i_{1,2},1,k_{1,2},\stackrel{T_1}{Z_1},\stackrel{T_0}{Z_0},\stackrel{T_2}{Z_2},\stackrel{\beta}{\Phi},\stackrel{\mathcal{E}}{\Upsilon}\right)=\stackrel{\mathcal{E}}{\Upsilon}\bigg(\stackrel{\mathcal{E}}{\Upsilon}\bigg(...\bigg(\stackrel{\mathcal{E}}{\Upsilon}\bigg(\stackrel{\mathcal{E}}{\Upsilon}(\omega,i),\Theta\bigg(E\left(\frac{k_{1,2}-1}{2^{\omega-1}}\right),i_{1,2},\stackrel{T_1}{Z_1},\stackrel{T_0}{Z_0},\stackrel{T_2}{Z_2},\stackrel{\beta}{\Phi},\stackrel{\mathcal{E}}{\Upsilon},0\bigg)\bigg),\\ \Theta\bigg(E\left(\frac{k_{1,2}-1}{2^{\omega-2}}\right)+1,i_{1,2},\stackrel{T_1}{Z_1},\stackrel{T_0}{Z_0},\stackrel{T_2}{Z_2},\stackrel{\beta}{\Phi},\stackrel{\mathcal{E}}{\Upsilon},0\bigg)\bigg),...\bigg),\Theta\bigg(E\left(\frac{k_{1,2}-1}{4}\right)+2^{\omega-3}-1,i_{1,2},\stackrel{T_1}{Z_1},\stackrel{T_0}{Z_0},\stackrel{T_2}{Z_2},\stackrel{\beta}{\Phi},\stackrel{\mathcal{E}}{\Upsilon},0\bigg)\bigg),\\ \Theta\bigg(E\left(\frac{k_{1,2}-1}{2}\right)+2^{\omega-2}-1,i_{1,2},\stackrel{T_1}{Z_1},\stackrel{T_0}{Z_0},\stackrel{T_2}{Z_2},\stackrel{\beta}{\Phi},\stackrel{\mathcal{E}}{\Upsilon},0\bigg)\bigg)
\end{multline}
\normalsize
\eqref{Theta_rel} implies
\footnotesize
\begin{equation}
\left\lbrace
\begin{array}{lc}
\Theta(E\left(\frac{k-1}{2^{\omega-1}}\right)+1,n,p,q,r,\alpha,\omega,0)=\Theta\bigg(E\left(\frac{k_{1,2}-1}{2^{\omega-1}}\right),i_{1,2},\stackrel{T_1}{Z_1},\stackrel{T_0}{Z_0},\stackrel{T_2}{Z_2},\stackrel{\beta}{\Phi},\stackrel{\mathcal{E}}{\Upsilon},0\bigg)&\text{ for $l=\omega-1$}\\
\Theta(E\left(\frac{k-1}{2^{\omega-2}}\right)+3,n,p,q,r,\alpha,\omega,0)=\Theta\bigg(E\left(\frac{k_{1,2}-1}{2^{\omega-2}}\right)+1,i_{1,2},\stackrel{T_1}{Z_1},\stackrel{T_0}{Z_0},\stackrel{T_2}{Z_2},\stackrel{\beta}{\Phi},\stackrel{\mathcal{E}}{\Upsilon},0\bigg)&\text{ for $l=\omega-2$}\\
. \\
.\\
.\\
\Theta(E\left(\frac{k-1}{4}\right)+2^{\omega-2}-1,n,p,q,r,\alpha,\omega,0)=\bigg(E\left(\frac{k_{1,2}-1}{4}\right)+2^{\omega-3}-1,i_{1,2},\stackrel{T_1}{Z_1},\stackrel{T_0}{Z_0},\stackrel{T_2}{Z_2},\stackrel{\beta}{\Phi},\stackrel{\mathcal{E}}{\Upsilon},0\bigg)&\text{ for $l=2$}\\
\Theta(E\left(\frac{k-1}{2}\right)+2^{\omega-1}-1,n,p,q,r,\alpha,\omega,0)=\Theta\bigg(E\left(\frac{k_{1,2}-1}{2}\right)+2^{\omega-2}-1,i_{1,2},\stackrel{T_1}{Z_1},\stackrel{T_0}{Z_0},\stackrel{T_2}{Z_2},\stackrel{\beta}{\Phi},\stackrel{\mathcal{E}}{\Upsilon},0\bigg)&\text{ for $l=1$}
\end{array}\right.
\end{equation}
\normalsize
Then, $\aleph_\omega(n,1,k,p,q,r,\alpha,\omega)=\aleph_\omega\left(i_{1,2},1,k_{1,2},\stackrel{T_1}{Z_1},\stackrel{T_0}{Z_0},\stackrel{T_2}{Z_2},\stackrel{\beta}{\Phi},\stackrel{\mathcal{E}}{\Upsilon}\right)$ so the relation \eqref{Sol1} is proven and the first equality \eqref{FirstEquality} is satisfied.\newline 
The second equality corresponds to
\footnotesize
\begin{multline}
\Gamma(\omega,p,q,r,\alpha,i)
\bigg[\prod\limits_{j_1=1}^{\omega-1}\prod\limits_{k_1=1}^{2^{j_1-1}}\Gamma(\rho_\omega(i_1,1,k_1,j_1,T_1,T_0,T_2,\beta,\mathcal{E}),\rho_1(i_1,1,k_1,j_1,T_1,T_0,T_2,\beta,\mathcal{E}),\\
\rho_0(i_1,1,k_1,j_1,T_1,T_0,T_2,\beta,\mathcal{E}),\rho_2(i_1,1,k_1,j_1,T_1,T_0,T_2,\beta,\mathcal{E}),\rho_\alpha(i_1,1,k_1,j_1,T_1,T_0,T_2,\beta,\mathcal{E}),\\ 
\Theta(k_1-2+2^{j_1-1},i_1,T_1,T_0,T_2,\beta,\mathcal{E},0))\bigg] 
 \bigg[\prod\limits_{j_2=1}^{\omega-1}\prod\limits_{k_2=1}^{2^{j_2-1}}
\Gamma(\rho_\omega(i_2,1,k_2,j_2,Z_1,Z_0,Z_2,\Phi,\Upsilon),\rho_1(i_2,1,k_2,j_2,Z_1,Z_0,Z_2,\Phi,\Upsilon),\\
\rho_0(i_2,1,k_2,j_2,Z_1,Z_0,Z_2,\Phi,\Upsilon), 
 \rho_2(i_2,1,k_2,j_2,Z_1,Z_0,Z_2,\Phi,\Upsilon),\rho_\alpha(i_2,1,k_2,j_2,Z_1,Z_0,Z_2,\Phi,\Upsilon),\\
 \Theta(k_2-2+2^{j_2-1},i_2,Z_1,Z_0,Z_2,\Phi,\Upsilon,0))\bigg] \\
 \stackrel{?}{=}
 \prod\limits_{j=1}^{\omega}\prod\limits_{k=1}^{2^{j-1}}\Gamma(\rho_\omega(n,1,k,j,p,q,r,\alpha,\omega),\rho_1(n,1,k,j,p,q,r,\alpha,\omega),\rho_0(n,1,k,j,p,q,r,\alpha,\omega), \\
 \rho_2(n,1,k,j,p,q,r,\alpha,\omega),\rho_\alpha(n,1,k,j,p,q,r,\alpha,\omega),\Theta(k-2+2^{j-1},n,p,q,r,\alpha,\omega,0))
\label{SecondEquality}
\end{multline}
\normalsize
First, there is the same number of terms in both sides of \eqref{SecondEquality}. Since $\sum\limits_{x=1}^y 2^{x-1}=2^y-1$, the right member is composed of $2^\omega -1$ terms. The left member is composed of $1+2(2^{\omega-1}-1)=2^\omega-1$ terms. \newline
Second, the term $k=1$ and $j=1$ in the right member corresponds to the individual term $\Gamma(\omega,p,q,r,\alpha,i)$ of the left member (direct application of the definition).\newline 
Third, the application of \eqref{Theta_rel} with $y=j-1$ leads to 
\footnotesize
\begin{equation}
\rho_x\left(n,1,k,j,p,q,r,\alpha,\omega\right)=
\left\lbrace
\begin{array}{lc}
\rho_x\left(i_1,1,k_1,j_1,T_1,T_0,T_2,\beta,\mathcal{E}\right)&\text{ if $k\leq 2^{j-2}$}\\
\rho_x\left(i_2,1,k_2,j_2,Z_1,Z_0,Z_2,\Phi,\Upsilon,0\right)&\text{ if $2^{j-2}<k$}\\
\end{array}\right.
\label{Sol2}
\end{equation}
\normalsize
with $k_1=k$, $k_2=k-2^{j-2}$ and $j_1=j_2=j-1\geq 1$. Like \eqref{Sol1}, the proof of \eqref{Sol2} is only presented for the case $x=\omega$:
\begin{multline}
\rho_\omega(n,1,k,j,p,q,r,\alpha,\omega)=\stackrel{\mathcal{E}}{\Upsilon}\bigg(\stackrel{\mathcal{E}}{\Upsilon}\bigg(...\bigg(\stackrel{\mathcal{E}}{\Upsilon}\bigg(\stackrel{\mathcal{E}}{\Upsilon}\bigg(\omega,\Theta\bigg(E\left(\frac{k-1}{2^{j-1}}\right),n,p,q,r,\alpha,\omega,0\bigg)\bigg),\\ \Theta\bigg(E\left(\frac{k-1}{2^{j-2}}\right)+1,n,p,q,r,\alpha,\omega,0\bigg)\bigg),...\bigg), \Theta\bigg(E\left(\frac{k-1}{4}\right)+2^{j-3}-1,n,p,q,r,\alpha,\omega,0\bigg)\bigg),\\
\Theta\bigg(E\left(\frac{k-1}{2}\right)+2^{j-2}-1,n,p,q,r,\alpha,\omega,0\bigg)\bigg)
\label{Int2}
\end{multline}
\\Since $k\leq 2^{j-1}$,
\begin{equation}
\begin{aligned}
E\left(\frac{k-1}{2^{j-1}}\right)=0\Rightarrow\Theta\bigg(E\left(\frac{k-1}{2^{j-1}}\right),n,p,q,r,\alpha,\omega,0\bigg)&=\Theta(0,n,p,q,r,\alpha,\omega,0)\\
&=\theta(n,p,q,r,\alpha,\omega,0,\omega-1)\\
&=i
\end{aligned}
\end{equation}
Then, \eqref{Int2} becomes
\footnotesize
\begin{multline}
\rho_\omega(n,1,k,j,p,q,r,\alpha,\omega)=\stackrel{\mathcal{E}}{\Upsilon}\bigg(\stackrel{\mathcal{E}}{\Upsilon}\bigg(...\bigg(\stackrel{\mathcal{E}}{\Upsilon}\bigg(\stackrel{\mathcal{E}}{\Upsilon}(\omega,i), \Theta\bigg(E\left(\frac{k-1}{2^{j-2}}\right)+1,n,p,q,r,\alpha,\omega,0\bigg)\bigg),...\bigg),\\ \Theta\bigg(E\left(\frac{k-1}{4}\right)+2^{j-3}-1,n,p,q,r,\alpha,\omega,0\bigg)\bigg),\Theta\bigg(E\left(\frac{k-1}{2}\right)+2^{j-2}-1,n,p,q,r,\alpha,\omega,0\bigg)\bigg)
\end{multline}
\normalsize
Regarding the other part of the equality \eqref{Sol2},
\scriptsize
\begin{multline}
\rho_\omega\left(i_{1,2},1,k_{1,2},j-1,\stackrel{T_1}{Z_1},\stackrel{T_0}{Z_0},\stackrel{T_2}{Z_2},\stackrel{\beta}{\Phi},\stackrel{\mathcal{E}}{\Upsilon}\right)=\stackrel{\mathcal{E}}{\Upsilon}\bigg(\stackrel{\mathcal{E}}{\Upsilon}\bigg(...\bigg(\stackrel{\mathcal{E}}{\Upsilon}\bigg(\stackrel{\mathcal{E}}{\Upsilon}(\omega,i),\Theta\bigg(E\left(\frac{k_{1,2}-1}{2^{j-2}}\right),i_{1,2},\stackrel{T_1}{Z_1},\stackrel{T_0}{Z_0},\stackrel{T_2}{Z_2},\stackrel{\beta}{\Phi},\stackrel{\mathcal{E}}{\Upsilon},0\bigg)\bigg),\\ \Theta\bigg(E\left(\frac{k_{1,2}-1}{2^{j-3}}\right)+1,i_{1,2},\stackrel{T_1}{Z_1},\stackrel{T_0}{Z_0},\stackrel{T_2}{Z_2},\stackrel{\beta}{\Phi},\stackrel{\mathcal{E}}{\Upsilon},0\bigg)\bigg),...\bigg),\Theta\bigg(E\left(\frac{k_{1,2}-1}{4}\right)+2^{j-4}-1,i_{1,2},\stackrel{T_1}{Z_1},\stackrel{T_0}{Z_0},\stackrel{T_2}{Z_2},\stackrel{\beta}{\Phi},\stackrel{\mathcal{E}}{\Upsilon},0\bigg)\bigg),\\ \Theta\bigg(E\left(\frac{k_{1,2}-1}{2}\right)+2^{j-3}-1,i_{1,2},\stackrel{T_1}{Z_1},\stackrel{T_0}{Z_0},\stackrel{T_2}{Z_2},\stackrel{\beta}{\Phi},\stackrel{\mathcal{E}}{\Upsilon},0\bigg)\bigg)
\end{multline}
\normalsize
The relation \eqref{Theta_rel} implies
\scriptsize
\begin{equation}
\left\lbrace
\begin{array}{lc}
\Theta(E\left(\frac{k-1}{2^{j-2}}\right)+1,n,p,q,r,\alpha,\omega,0)=\Theta\bigg(E\left(\frac{k_{1,2}-1}{2^{j-2}}\right),i_{1,2},\stackrel{T_1}{Z_1},\stackrel{T_0}{Z_0},\stackrel{T_2}{Z_2},\stackrel{\beta}{\Phi},\stackrel{\mathcal{E}}{\Upsilon},0\bigg)&\text{ for $l=j-2$}\\
\Theta(E\left(\frac{k-1}{2^{j-3}}\right)+3,n,p,q,r,\alpha,\omega,0)=\Theta\bigg(E\left(\frac{k_{1,2}-1}{2^{j-3}}\right)+1,i_{1,2},\stackrel{T_1}{Z_1},\stackrel{T_0}{Z_0},\stackrel{T_2}{Z_2},\stackrel{\beta}{\Phi},\stackrel{\mathcal{E}}{\Upsilon},0\bigg)&\text{ for $l=j-3$}\\
. \\
.\\
.\\
\Theta(E\left(\frac{k-1}{4}\right)+2^{j-3}-1,n,p,q,r,\alpha,\omega,0)=\bigg(E\left(\frac{k_{1,2}-1}{4}\right)+2^{j-4}-1,i_{1,2},\stackrel{T_1}{Z_1},\stackrel{T_0}{Z_0},\stackrel{T_2}{Z_2},\stackrel{\beta}{\Phi},\stackrel{\mathcal{E}}{\Upsilon},0\bigg)&\text{ for $l=2$}\\
\Theta(E\left(\frac{k-1}{2}\right)+2^{j-2}-1,n,p,q,r,\alpha,\omega,0)=\Theta\bigg(E\left(\frac{k_{1,2}-1}{2}\right)+2^{j-3}-1,i_{1,2},\stackrel{T_1}{Z_1},\stackrel{T_0}{Z_0},\stackrel{T_2}{Z_2},\stackrel{\beta}{\Phi},\stackrel{\mathcal{E}}{\Upsilon},0\bigg)&\text{ for $l=1$}
\end{array}\right.
\end{equation}
\normalsize
Then, $\rho_\omega(n,1,k,j,p,q,r,\alpha,\omega)=\rho_\omega\left(i_{1,2},1,k_{1,2},j-1,\stackrel{T_1}{Z_1},\stackrel{T_0}{Z_0},\stackrel{T_2}{Z_2},\stackrel{\beta}{\Phi},\stackrel{\mathcal{E}}{\Upsilon}\right)$. So, there is the following association between the right and the left members of \eqref{SecondEquality}: 
\begin{center}
\begin{tabular}{cc}
Right member & Left member \\
$j=1,~k=1$ & First term \\
$j=2,~k=1$ & $j_1=1,~k_1=1$ \\
$j=2,~k=2$ & $j_2=1,~k_2=1$ \\
$j=3,~k=1$ & $j_1=1,~k_1=1$ \\
$j=3,~k=2$ & $j_1=2,~k_2=2$ \\
$j=3,~k=3$ & $j_2=1,~k_1=1$ \\
$j=3,~k=4$ & $j_2=2,~k_2=2$ \\
$j=4$ ...
\end{tabular}
\end{center}
The second equality \eqref{SecondEquality} is then satisfied and the solution \eqref{ucoef} holds $\forall\omega\in\mathbb{N}^*$. The proof of the theorem is then completed.

\newpage

\bibliographystyle{plain}
\bibliography{Biblio}




%
%

\end{document}